\documentclass[12pt]{amsart}

\allowdisplaybreaks[4] 
\usepackage{amsmath,amsthm,amscd,amssymb,amsfonts, amsbsy}
\usepackage{latexsym, color}
\usepackage{pxfonts}
\usepackage{mathrsfs}
\usepackage{bbm}
\usepackage{mathtools}

\usepackage{lineno,hyperref}
\modulolinenumbers[5]

\usepackage{enumitem}

\usepackage{todonotes}

\newtheorem{theorem}{Theorem}
[section]%
\newtheorem{proposition}[theorem]{Proposition}%
\newtheorem{lemma}[theorem]{Lemma}%

\theoremstyle{remark}%
\newtheorem{remark}{Remark}%

\theoremstyle{definition}%

\newcommand{\dv}{\operatorname{div}}

\numberwithin{equation}{section}

\newcommand{\bN}{\mathbb{N}}

\newcommand{\bR}{\mathbb{R}}

\newcommand{\lomu}{\underline{\mu}}

\newcommand{\al}{\alpha}
\newcommand{\be}{\beta}
\newcommand{\ep}{\epsilon}
\newcommand{\lt}{\left}
\newcommand{\rt}{\right}

\newcommand\cC{\mathcal{C}}

\newcommand\cI{\mathcal{I}}
\newcommand\cJ{\mathcal{J}}

\newcommand\cO{\mathcal{O}}

\newcommand\te{\theta}

\newcommand{\ud}{\textnormal{d}}
\newcommand{\dt}{\ud t}

\providecommand{\norm}[1]{\lVert#1\rVert}

\newcommand{\Lam}{\Lambda}

\newcommand{\p}{\partial}

\begin{document}
\title[]
{Asymptotic stability of the $2D$ temperature-dependent tropical climate model with the sharp decay rates}

\author[H. In]{Hyunjin In}
\address[Hyunjin In]{Research Institute of Basic Sciences, Ajou University 206, World cup-ro, Yeongtong-gu, Suwon-si, Gyeonggi-do, Republic of Korea}
\email{hyunjinin@ajou.ac.kr}

\author[D. Kim]{Dong-ha Kim}
\address[Dong-ha Kim]{CAU Nonlinear PDE center, Chung-Ang University, 84 Heukseok-ro, Dongjak-gu, Seoul 06974, Republic of Korea}
\email{kimdongha91@cau.ac.kr}

\author[J. Kim]{Junha Kim}
\address[Junha Kim]{Department of mathematics, Ajou University 206, World cup-ro, Yeongtong-gu, Suwon-si, Gyeonggi-do, Republic of Korea}
\email{junha02@ajou.ac.kr}


\begin{abstract}
We investigate the asymptotic stability of a tropical climate model posed on $\bR^2$, with temperature-dependent diffusion in the barotropic mode $u$ and linear damping in the first baroclinic mode $v$.
We consider two distinct cases for the barotropic component: one with linear damping and one without. 
For both cases, we prove the small data global existence of smooth solutions. Furthermore, we establish sharp temporal decay estimates for solutions in arbitrary Sobolev norms $H^m (\bR^2)$, $m \ge 0$.
\end{abstract}

\maketitle
\tableofcontents
\section{Introduction and main results}
\label{S1}

Frierson--Majda--Pauluis \cite{MR2119930} introduced the original version of the tropical climate model in $\bR^2$:
\begin{equation}
\label{E10}
\begin{cases}
    u_t + (u \cdot \nabla) u + \nabla p 
    = - \dv (v \otimes v), \\
    v_t + (u \cdot \nabla) v + (v \cdot \nabla) u = \nabla \theta, \\
    \theta_t + u \cdot \nabla \theta = \dv v,\\
    \dv u =0.
\end{cases}
\end{equation}
Here, $u$ and $v$ are the barotropic and the first baroclinic modes of the velocity field, respectively.
$\theta$ represents the temperature, and $p$ is the pressure.

Since the system \eqref{E10} is used to describe various physics phenomena, such as atmospheric and oceanic dynamics, it has attracted considerable interest from mathematicians.
Physically, the atmospheric dynamics in the Earth's tropics are largely characterized by the first baroclinic mode, where winds in the lower and upper troposphere are of equal magnitude but opposite sign. While this mode captures the primary vertical structure, a comprehensive study of phenomena like tropical-extratropical interactions requires accounting for the momentum exchange between the baroclinic component and the barotropic mode (the vertical average). Therefore, retaining both the barotropic and baroclinic modes is essential for a physically complete model (see \cite[P. 957]{MR3916791}).


In this paper, we consider the system \eqref{E10} with the diffusion and the linear damping terms in the barotropic mode $u$, the linear damping term in the first baroclinic mode $v$: 
\begin{align}
\label{E11}
\begin{cases}
    u_t + (u \cdot \nabla) u + \nabla p  -\dv(\mu(\theta)\nabla u)  +\alpha u
    = - \dv (v \otimes v), \\
    v_t + (u \cdot \nabla) v + (v \cdot \nabla) u +\beta v= \nabla \theta, \\
    \theta_t + u \cdot \nabla \theta = \dv v, \\
    \dv u =0.
\end{cases}
\end{align}
Here $\alpha\geq0$ is the nonnegative damping parameter and $\beta>0$ is the positive damping parameter. 
The viscosity coefficient of the barotropic mode, denoted by $\mu(\theta)$, is the composition of temperature $\theta$ and some smooth function $\mu: \bR \rightarrow \bR^+$ which satisfies
\[
\mu(\cdot) \in \cC^k( \bR) \quad \text{for all } \,k > 1,
\]
and has a positive lower bound $ \underline{\mu} > 0$.

Mathematically, for the $2D$ case, Li and Titi in \cite{MR3479523} established the global well-posedness of \eqref{E10} with diffusion in $u$ and $v$.
To overcome the absence of the thermal diffusion, they utilized a new unknown $\omega = v - \nabla (- \Delta)^{-1} \theta$.
In \cite{MR3455664}, Wan established the global well-posedness for \eqref{E10} with diffusion term $u$ and linear damping term $u$, $v$ under small initial data in 
$\norm{u_0}_{H^s (\bR^2)}
+ \norm{v_0}_{H^s (\bR^2)}
+ \norm{\theta_0}_{H^s (\bR^2)}$ with $s > 2$.
Later, Ma and Wan \cite{MR3698155} proved the global well-posedness of \eqref{E10} with diffusion term $u$ and linear damping term $v$ under small initial data 
$\norm{u_0}_{H^s (\bR^2) \cap \dot{B}_{2,1}^0 (\bR^2)}
+ \norm{v_0}_{H^s (\bR^2)}
+ \norm{\theta_0}_{H^s (\bR^2)}$ with $s > 3$.
In \cite{MR4495491}, Niu and Wu studied the global well-posedness and large-time behavior of solutions to the system \eqref{E10} with diffusion term $u$ and linear damping term $v$.

For temperature-dependent diffusion in \eqref{E11}, Ye and Zhu \cite{MR3886970} proved the existence of the global strong solutions to the system \eqref{E11} with $\alpha, \beta > 0$ under the initial data 
$\norm{u_0}_{H^s (\bR^2)}
+ \norm{v_0}_{H^s (\bR^2)}
+ \norm{\theta_0}_{H^s (\bR^2)}$ with $s > 1$ is suitably small.
In \cite{MR4106813}, Ye and Zhu proved the global strong solutions to
the system \eqref{E11} in $\beta = 0$ with temperature-dependent diffusion $u$ and $v$ provided that the initial data $\norm{v_0}_{H^s (\bR^2)}
+ \norm{\theta_0}_{H^s (\bR^2)}$ with $s > 1$ is small enough.
Moreover, several results on different models of temperature-dependent diffusion in \eqref{E11} can be found in \cite{MR4281486,MR4385767}).
For the system \eqref{E10} with fractional order dissipation, we refer the reader to \cite{MR3554729,MR3932724,MR3927105,MR3926071,MR4084151,MR4112846,MR4135783,MR4228545}
and the references therein.
For other results in various types of \eqref{E10}, we refer the reader to
\cite{MR4151388,MR4218050,MR4328431,MR4583541,MR4844575}
and the references therein.

In this paper, we investigate the asymptotic stability analysis and derive the temporal decay estimates for the system~\ref{E11}, with particular attention to the effect of the linear damping term $\alpha u$ by dividing it into two distinct cases: either $\alpha = 0$ or $\alpha > 0$. Each case would be analyzed separately to highlight the differences. 
In particular, assuming that a certain Sobolev norm of the initial data is suitably small, we establish asymptotic stability and temporal decay estimates in a Sobolev space of higher regularity.
Since we consider the temperature $\theta$ in the absence of thermal diffusion, the damping term for $\theta$ does not appear in the energy estimates.
To overcome this difficulty, we are inspired by \cite{MR3886970,MR4106813}.
A key component of our approach is to utilize the intrinsic structure of the model to derive a damping term for $\theta$ in the energy estimates.
This term plays a crucial role in the proof of our main results.

We are now ready to state the main theorems. To state more neatly, we define a positive constant $\lambda = \lambda(\alpha, \beta, \underline{\mu})$ by
\begin{equation}
\label{lambda}
\lambda :=
\begin{cases}
\left(\frac{1}{2}\min(\underline{\mu}, \frac{\beta}{4+2\beta^2})\right)^{\frac{1}{2}} &\quad \mbox{if }\al =0,\\
    \left(\frac{1}{2}\min(\al, \underline{\mu}, \frac{\beta}{4+2\beta^2})\right)^{\frac{1}{2}} &\quad \mbox{if }\al >0.
    \end{cases}
\end{equation} 

Here are our main results.

\begin{theorem}[Undamped case]\label{main_thm:no_damping} Let $\alpha =0$ and $s>1$.
There exists a positive number $\ep_1$, depending only on $s, \beta, \lomu$, and $\|\mu'\|_{C^{s-1}}$, which satisfies the following: 
let the initial data $(u_0,v_0,\theta_0)$ belong to $\mathcal{C}_0^{\infty} (\mathbb{R}^2)$, and assume that $\dv u_0=0$. Suppose that $0<\epsilon<\epsilon_1$ and
\begin{equation}
\label{E7}
\|u_0\|_{H^s(\bR^2)}+\|v_0\|_{H^s(\bR^2)}+\|\te_0\|_{H^s(\bR^2)} < \ep.
\end{equation}
Then, the Cauchy problem \eqref{E11} has a unique global solution $(u,v,\te) \in \mathcal{C}^{\infty} ([0,\infty); \cC^{\infty} (\bR^2))$. Indeed, for each $m \ge s$, there exists a constant $C$ which depends only on $m, s, \beta$, and $\|\mu'\|_{C^{m-1}(\bR)}$ such that
\begin{equation}
\label{E15}
\begin{split}
\sup_{t \in [0,\infty)}\left(\|u(t)\|_{H^m(\bR^2)}
+\|v(t)\|_{H^m(\bR^2)}
+\|\te(t)\|_{H^m(\bR^2)} \right)
+ 
&\lambda
\left( \int_0^{\infty}  
\| \nabla u(t)\|_{H^m (\bR^2)}^2
+ \|v(t)\|_{H^m(\bR^2)}^2
+ \| \nabla \te(t)\|_{H^{m-1} (\bR^2)}^2
\,\ud t \right)^{\frac{1}{2}} \\ 
&\qquad \leq C( \|u_0\|_{H^m(\bR^2)}+\|v_0\|_{H^m(\bR^2)}+\|\te_0\|_{H^m(\bR^2)}).
\end{split}
\end{equation}

Moreover, for each positive number $\gamma$, there holds that
\begin{equation}
\label{E44}
(1 + t)^{\frac{\gamma}{2}}\norm{\Lambda^{\gamma} u(t)}_{L^2(\bR^2)}
+ (1 + t)^{\frac{\gamma+1}{2}}\norm{\Lambda^{\gamma} v(t)}_{L^2(\bR^2)}
+ (1 + t)^{\frac{\gamma}{2}}\norm{\Lambda^{\gamma} \theta(t)}_{L^2(\bR^2)}
\le C ,
\end{equation}
where the constant $C$ depends on $\gamma,s,\beta,\lomu,\|\mu'\|_{C^{\gamma+1}(\bR)}$, and $\|(u_0,v_0,\te_0)\|_{H^{\gamma+1}(\bR^2)}$.
\end{theorem}

\begin{theorem}[Damped case]\label{main_thm:damped} Let $\alpha >0$ and $s>1$. 
There exists a positive number $\ep_2$, depending only on $s, \alpha, \beta, \lomu$, and $\|\mu'\|_{C^{s-1}}$, which satisfies the following: let the initial data $(u_0,v_0,\theta_0)$ belong to $\cC^{\infty}_0 (\mathbb{R}^2)$, and assume that $\dv u_0 =0$. Suppose that $0<\epsilon<\epsilon_2$ and
\begin{equation}\label{E8}
\|u_0\|_{\dot{H}^s\cap \dot{H}^1(\bR^2)}+\|v_0\|_{H^s(\bR^2)}+\|\te_0\|_{H^s(\bR^2)} < \ep.
\end{equation}
Then, the Cauchy problem \eqref{E11} has a unique global solution $(u,v,\te) \in \mathcal{C}^{\infty} ([0,\infty); \cC^{\infty} (\bR^2))$. For each $m \ge s$, there exists a constant $C$ which depends only on $m, s, \al, \beta$, and $\|\mu'\|_{C^{m-1}(\bR)}$ such that
\begin{equation}
\label{E16}
\begin{split}
\sup_{t \in [0,\infty)}
\left(
\|u(t)\|_{H^m (\bR^2)}
+\|v(t)\|_{H^m(\bR^2)}
+\|\te(t)\|_{H^m(\bR^2)} \right)
+
&\lambda
\left( \int_0^{\infty}  
\|u(t)\|_{H^{m+1}(\bR^2)}^2
+ \|v(t)\|_{H^m(\bR^2)}^2
+ \| \nabla \te(t)\|_{H^{m-1} (\bR^2)}^2  
\,\ud t \right)^{\frac{1}{2}} \\ 
&\quad \leq C( 
\|u_0\|_{H^m(\bR^2)}
+ \|v_0\|_{H^m(\bR^2)}
+ \|\te_0\|_{H^m(\bR^2)}).
\end{split}
\end{equation}

Moreover, for each positive number $\gamma$, there holds that
\begin{equation}
\label{E45}
(1 + t)^{\frac{\gamma+4}{2}}\norm{\Lambda^{\gamma} u(t)}_{L^2(\bR^2)}
+ (1 + t)^{\frac{\gamma+1}{2}}\norm{\Lambda^{\gamma} v(t)}_{L^2(\bR^2)}
+ (1 + t)^{\frac{\gamma}{2}}\norm{\Lambda^{\gamma} \theta(t)}_{L^2(\bR^2)}
\le C ,
\end{equation}
where the constant $C$ depends on $\gamma,s,\alpha, \beta, \|\mu\|_{C^{\gamma+1}(\bR)}$, and $\|(u_0,v_0,\te_0)\|_{H^{\gamma+1}(\bR^2)}$.
\end{theorem}

\begin{remark}
For the dependence of the quantities, we emphasize that
    \begin{enumerate}
        \item[(i)] The constant $\epsilon_1$ and $\epsilon_2$ does not depend on $m$,
        \item[(ii)] 
        The constant $C$ in \eqref{E15} and \eqref{E16} can be represented by $C=1+\cO(\epsilon)$, where $\cO(\epsilon)$ indicates that there exists a some positive constant $C'=C'(m, s, \alpha, \beta, \|\mu'\|_{C^{m-1}(\bR)})$ such that
        \begin{equation*}
            \lt|\,\cO(\epsilon)\,\rt| \leq C'\ep.
        \end{equation*} Hence, the constant $C$ tends to $1$, as $\epsilon \to 0+$. See the proof of Proposition~\ref{P33} for the details.
\end{enumerate}
\end{remark}

\begin{remark}
Let $m\geq s>1$. If $(u_0, v_0, \theta_0) \in H^m(\mathbb{R}^2)$ satisfies \eqref{E7} in the undamped case (or \eqref{E8} in the damped case), then we can find a global solution 
$(u, v, \theta) \in \mathcal{C}([0,\infty); H^m(\mathbb{R}^2))$ satisfying \eqref{E15} (or \eqref{E16}, respectively). Regarding temporal decay estimates, the range of $\gamma$ for which an estimate of the same form holds is limited to $\gamma \in [0, m-1]$.
\end{remark}

\begin{remark}
\begin{enumerate}
\item
Our Theorem \ref{main_thm:no_damping} for the system \eqref{E11} with $\mu (\theta) = Const. > 0$ and $\alpha = 0$ relaxes the initial data requirement of \cite[Theorem 1.1]{MR3698155}, which assumes smallness of $\norm{u_0}_{H^s (\bR^2) \cap \dot{B}_{2,1}^0 (\bR^2)}
+ \norm{v_0}_{H^s (\bR^2)}
+ \norm{\theta_0}_{H^s (\bR^2)}$ with $s > 3$. 
\item
Our Theorem \ref{main_thm:damped} relax the condition of \cite[Theorem 1.1]{MR3886970}.
\end{enumerate}
\end{remark}

\begin{remark}
The sharpness of the decay rates for $v$ and $\theta$ in Theorem~\ref{main_thm:damped} and Theorem~\ref{main_thm:no_damping} can be shown in a similar manner to that in \cite[Section 7]{MR4581689}.
In the damped case $\al>0$, we further establish improved decay for $u$, including a decay rate for its $L^2$-norm.
\end{remark}

We organize this paper as follows.
In Section \ref{S2}, we introduce several key technical tools that include various forms of the Gagliardo--Nirenberg inequality, the Kato--Ponce commutator estimates, and Moser-type calculus inequalities. 
We also mention the local-in-time wellposedness.
In Section \ref{S3}, we establish the asymptotic stability and the temporal decay estimates of the system \eqref{E11} by assuming sufficiently small initial data in Sobolev spaces.
In Section \ref{S4}, we present the proof of our main results in higher Sobolev spaces just assuming small initial data in lower Sobolev spaces.
In analogy with Section \ref{S3}, the cross term plays the crucial role in the proof of the main results.

\section{Preliminaries}
\label{S2}

For notational convenience, we introduce a few notations.

\begin{itemize}
\item
We denote $A \lesssim B$ if there exists a generic positive constant $C$ such that $|A| \le C B$.
\item
The letter $C$ represents a generic constant, which may change from
line to line.
\item
For smooth $f$, we denote $\Lambda^s f$ by a smooth function satisfying $\widehat{\Lambda^s f}(\xi) = |\xi|^s\hat{f}(\xi)$ for $\xi \in \bR^2$. We shall use the natural $L^p$-extension of $\Lambda^s f$.
\item
For $s > 0$, we define the homogeneous Sobolev space $\dot H^s = \dot H^s (\bR^2)$ by $\norm{f}_{\dot H^s (\bR^2)}^2 
= \norm{\Lam^s f}_{L^2 (\bR^2)}^2$ and the nonhomogeneous Sobolev space $H^s = H^s (\bR^2)$ by $\norm{f}_{H^s (\bR^2)}^2 
= \norm{f}_{L^2 (\bR^2)}^2 
+ \norm{\Lam^s f}_{L^2 (\bR^2)}^2$.
\item
We simply write $\norm{f}_{L^p} = \norm{f}_{L^p ( \bR^2)}$ for $p \in [1,\infty]$ and $\norm{f}_{H^s} = \norm{f}_{H^s ( \bR^2)}$ for $s > 0$.
Similarly, we write
$\norm{f}_{\dot H^{s_1} (\bR^2)
\cap \dot H^{s_2} (\bR^2)}^2 
= \norm{\Lam^{s_1} f}_{L^2 (\bR^2)}^2
+ \norm{\Lam^{s_2} f}_{L^2 (\bR^2)}^2$ for $s_1, s_2 > 0$.
\item  
For vector fields $a,b$, we write $(a \otimes b)$ as $(a_i b_j)_{1 \le i,j \le 2}$.
We denote $A : B = a_{ij} b_{ij}$ for $2 \times 2$ matrices $A = (a_{ij})$ and $B = (b_{ij})$.
Here, we use the Einstein summation convention.
\end{itemize}

\begin{lemma}[Gagliardo--Nirenberg]
\label{L22}
For $1 \le q,r \le \infty$ and $ \gamma,j,m$ satisfying
\[
\frac{1}{p}
= \frac{j}{d} + \gamma \left( \frac{1}{r} - \frac{m}{d} \right) + \frac{1- \gamma}{q}
\]
with
\[
\frac{j}{m} \le \gamma \le 1,
\]
there exists some generic constant $C > 0$ which may depend on $p,q,r,j,m, \gamma$ such that $f \in L^q ( \bR^d)$, $\Lam^j f \in L^p ( \bR^d)$, and $\Lam^m f \in L^r ( \bR^d)$, we have
\begin{equation}
\label{E21}
\norm{ \Lam^j f}_{L^p(\bR^d)}
\le C \norm{ \Lam^m f}_{L^r(\bR^d)}^\gamma
\norm{f}_{L^q(\bR^d)}^{1- \gamma}.
\end{equation}
In particular, by choosing $j = 0$ and $ \gamma = 1$, we can obtain that 
\begin{equation}
\label{E22}
\norm{f}_{L^p(\bR^d)}
\le C \norm{ \Lam^m f}_{L^r(\bR^d)},
\end{equation}
where $p = \frac{dr}{d-rm}$.
\end{lemma}

\begin{remark} In practice, for Lemma~\ref{L22}, we primarily make use of the following forms
\begin{equation}
\begin{gathered}
\label{E14}
\|f\|_{L^4(\bR^2)} \leq C \|\Lambda^{\frac{1}{2}}f\|_{L^2(\bR^2)},
\\
\|\Lam^{s-1}f\|_{L^{\frac{2}{s-1}}(\bR^2)}\leq C \|\nabla f\|_{L^2(\bR^2)}
\quad 
\text{and} 
\quad \|\nabla f\|_{L^{\frac{2}{2-s}}(\bR^2)}\leq C \|\Lam^s f\|_{L^2(\bR^2)}
\quad
\text{for $1 < s < 2$}.
\end{gathered}
\end{equation}
\end{remark}
\begin{lemma}[$L^{\infty}$-bound, $\Lambda^s$-type interpolation.] \label{L23} Let $d$ be a positive integer.
    \begin{enumerate}
        \item[(i)] Let $s_1,s_2$ and $s$ be nonnegative real numbers satisfying $0\leq s_1<s<s_2$.
    Suppose $u \in H^{s_2}(\mathbb{R}^d)$. Then, it holds that
\[
        \|\Lambda^su\|_{L^2(\bR^d)} \leq  \|\Lambda^{s_1}u\|_{L^2(\bR^d)}^{\frac{s_2-s}{s_2-s_1}} \|\Lambda^{s_2}u\|_{L^2(\bR^d)}^{\frac{s-s_1}{s_2-s_1}}.
\]
    \item[(ii)]  
    Let $s\geq0$, and let $s_1,s_2$ be two nonnegative real numbers satisfying $0\leq s_1<\frac{d}{2}$ and $\frac{d}{2}+s<s_2$.
Suppose $u \in H^{s_2}(\mathbb{R}^d)$. Then, there exists a constant $C=C(d,s,s_1,s_2)>0$  such that
\[
        \|\Lambda^s u\|_{L^{\infty}(\bR^d)} \leq C \|\Lambda^{s_1}u\|_{L^2(\bR^d)}^{\frac{s_2-s-\frac{d}{2}}{s_2-s_1}} \|\Lambda^{s_2}u\|_{L^2(\bR^d)}^{\frac{\frac{d}{2}+s-s_1}{s_2-s_1}}.
\]
In particular, there holds
\[
\|u\|_{L^{\infty}(\bR^d)} \leq C \|\Lambda^{s_1}u\|_{L^2(\bR^d)}^{\frac{s_2-\frac{d}{2}}{s_2-s_1}} \|\Lambda^{s_2}u\|_{L^2(\bR^d)}^{\frac{\frac{d}{2}-s_1}{s_2-s_1}}.
\]
    \end{enumerate}
\end{lemma}
    \begin{proof}
We refer to Appendix~\ref{Appen_B}.
    \end{proof}

\begin{lemma}[Kato--Ponce \cite{MR1086966,MR951744,MR3914540}]
\label{L24}
Let $s > 0$ and $1 < p < \infty$.
If $\nabla f \in L^{p_1} ( \bR^d)$, $ \Lam^{s-1} g \in L^{p_2} ( \bR^d)$, $g \in L^{p_3} ( \bR^d)$, and $ \Lam^s f \in L^{p_4} ( \bR^d)$, then
\begin{equation}
\label{E23}
\norm{ \Lam^s (fg)}_{L^p(\bR^d)}
\le C \left( \norm{f}_{L^{p_1}(\bR^d)} 
\norm{ \Lam^{s} g}_{L^{p_2}(\bR^d)}
+ \norm{g}_{L^{p_3}(\bR^d)}
\norm{ \Lam^s f}_{L^{p_4}(\bR^d)} \right),
\end{equation}
\begin{equation}
\label{E24}
\norm{ \Lam^s (fg) - f \Lam^s g}_{L^p(\bR^d)}
\le C \left( \norm{ \nabla f}_{L^{p_1}(\bR^d)} 
\norm{ \Lam^{s-1} g}_{L^{p_2}(\bR^d)}
+ \norm{g}_{L^{p_3}(\bR^d)}
\norm{ \Lam^s f}_{L^{p_4}(\bR^d)} \right),
\end{equation}
where $p_2, p_4 \in (1, \infty)$ and $p_1, p_3 \in [1, \infty]$ such that $\frac{1}{p} = \frac{1}{p_1} + \frac{1}{p_2}
= \frac{1}{p_3} + \frac{1}{p_4}$ and $C > 0$ is a constant depending only on $p_1, p_2, p_3, p_4, s$, and $d$.
\end{lemma}

We present Moser-type calculus inequality.
\begin{lemma}
\label{L25}
Let $s \ge 1$.
Suppose that $h:\Omega \rightarrow \bR$ is a smooth real-valued function on $\Omega$ with $h(0)=0$. We also assume that $f : \bR^{d} \rightarrow \Omega$ is a continuous function with $\text{Im}(f) \subset \Omega_1$, $\overline{\Omega}_1 \subset \subset \Omega$ and $f \in L^\infty \cap H^s$. We denote the composition of $h$ and $f$ by $h(f)$. Then, the composition $h (f)$ belongs to $H^s (\bR^d)$, and it holds that
\[
\norm{ \Lam^s h (f)}_{L^2(\bR^d)}
\le C
\norm{h'}_{C^{\lceil s-1 \rceil} (\overline{ \Omega_1})} 
(1 + \norm{\nabla f}_{L^2(\bR^d)}^{\lceil s-1 \rceil})
\norm{ \Lam^s f}_{L^2(\bR^d)},
\]
where $C$ is a constant depending only on $s\geq1$ and $d$.
\end{lemma}
    \begin{proof}
We refer to Appendix~\ref{Appen_B}. See also \cite[Proposition 2.1]{M1984}.
    \end{proof}

\begin{remark} In practice, for Lemma~\ref{L25}, the primary forms we utilize is as follows:
\begin{equation}
\label{E70}
\|\Lambda^s(\mu(\te)-\mu(0))\|_{L^2(\bR^2)} \leq C\norm{\mu'}_{C^{\lceil s-1 \rceil}} 
(1 + \norm{\nabla \te}_{L^2(\bR^2)}^{\lceil s-1 \rceil})
\norm{ \Lam^s \te}_{L^2(\bR^2)} \leq C 
(1 + \norm{\nabla \te}_{L^2(\bR^2)}^{\lceil s-1 \rceil})
\norm{ \Lam^s \te}_{L^2(\bR^2)}.
\end{equation}
On the other hands, one can simply use the chain rule, $\nabla \mu(\te) = \mu'(\te)\nabla \te$, to see that
\begin{equation}
\label{E71}
\|\nabla (\mu(\te)-\mu(0))\|_{L^p(\bR^2)} = \|\nabla \mu(\te)\|_{L^p(\bR^2)} \leq \norm{\mu'}_{L^{\infty}}  
\norm{ \nabla \te}_{L^p(\bR^2)} 
\leq C 
\norm{ \nabla \te}_{L^p(\bR^2)}\quad \text{for all } p \in [1,\infty].
\end{equation}
We shall visit these later.
\end{remark}

Here, we introduce the local existence and uniqueness of strong solutions, which can be proved by standard way. We omit the proof.
\begin{theorem}[Local-in-time existence]
For $m>1$, let $u_0,v_0$, and $\te_0$ be in $H^m(\bR^2)$ with $\dv u_0=0$. Then, there exists $T>0$ such that the system of the tropical climate model \eqref{E11} possesses a unique solution $(u,v,\te)$ in $\mathcal{C}([0,T];H^m(\bR^2))$ subject to initial data $(u_0, v_0, \te_0)$. 
\end{theorem}

\section{Key proposition}
\label{S3}
In this section, we establish the stability and temporal decay estimates in $H^s(\bR^2)$ for $s>1$. To this end, we temporarily assume that the initial data belongs to $H^s(\bR^2)$. 
The following proposition is the main result of this section and serves as a crucial step in the proof of the main theorem.
\begin{proposition}
\label{key_prop}
Let $\al \geq 0$ and $s>1$, and suppose that the initial data $(u_0,v_0,\theta_0)$ belong to $H^s(\mathbb{R}^2)$. We assume that $\dv u_0$=0. Then  there exists a positive number $\varepsilon^*>0$ satisfying the following:
\begin{enumerate}
\item (Undamped case) In case of $\alpha =0$, if there hold $0<\epsilon<\varepsilon^*$ and
\begin{equation}
\label{Assump1}
\|u_0\|_{H^s}+\|v_0\|_{H^s}+\|\te_0\|_{H^s} < \epsilon,
\end{equation}
then the Cauchy problem \eqref{E11} has a unique global solution $(u,v,\te) \in \cC([0,\infty),H^s(\bR^2))$ satisfying
\begin{equation}\label{small:no_damping}
\sup_{t \in [0,\infty)}\left(\|u(t)\|_{H^s}
+\|v(t)\|_{H^s}
+\|\te(t)\|_{H^s} \right)+
\lambda
\left(
\int_0^{\infty} 
\| \nabla u(t)\|_{H^s}^2
+\|v(t)\|_{H^s}^2
+ \|\nabla \te(t)\|_{H^{s-1}}^2 \,\ud t\right)^{\frac{1}{2}} < 2\epsilon.
\end{equation}
    \item (Damped case) In case of $\alpha >0$, if there hold $0<\epsilon<\varepsilon^*$ and 
\begin{equation}
\label{Assump2}
\|u_0\|_{\dot{H}^s\cap\dot{H}^1}+\|v_0\|_{H^s}+\|\te_0\|_{H^s} < \epsilon,
\end{equation}
then the Cauchy problem \eqref{E11} has a unique global solution $(u,v,\te) \in \cC([0,\infty),H^s(\bR^2))$ satisfying
\begin{equation} 
\label{small:damped}
    \sup_{t \in [0,\infty)}\left( \|u(t)\|_{\dot{H}^s\cap\dot{H}^{1}}
+\|v(t)\|_{H^s}
+\|\te(t)\|_{H^s} \right)
+\lambda\left(\int_0^{\infty} 
\|\nabla u(t)\|_{H^s}^2 
+\|v(t)\|_{H^s}^2 
+\|\nabla \te(t)\|_{H^{s-1}}^2\,\ud t \right)^{\frac{1}{2}}  < 2\epsilon.
\end{equation}
\end{enumerate}
Here, $\lambda$ is as defined in \eqref{lambda}.
Let us further assume that $s \in(1,2)$. Then, in both cases, there exists a positive constant $C>0$, depending on $s,\al,\be,\lomu, \|\mu'\|_{C^{s-1}}$, and $\|(u_0,v_0,\te_0)\|_{H^s}$ such that
\begin{equation}
\label{s-1:decay}
\norm{\Lambda^{s-1} u(t)}_{L^2}
+ \norm{\Lambda^{s-1} v(t)}_{L^2}
+ \norm{\Lambda^{s-1} \theta(t)}_{L^2}
\le C (1 + t)^{- \frac{s-1}{2}},\qquad \text{for all }t>0.
\end{equation}
\end{proposition}

These decay rates \eqref{s-1:decay} are not sharp and can be improved (see \eqref{E44} and \eqref{E45}).

\vspace{1.2em}

The proof of this proposition will be given in subsection~\ref{subsec_33}.

\subsection{An a priori estimate for stability}
The goal of this section is to state and prove the Lemma~\ref{L31}, which concerns an a priori estimate. This will be used to establish Proposition~\ref{key_prop}.

We define a number ${\delta_1}$ by
    \begin{equation}
    \label{delta_1}
        {\delta_1} =
        \begin{cases}
        0 & \text{if } \alpha =0,\\
        1 & \text{if } \alpha >0.    
        \end{cases}
    \end{equation}
We introduce the number ${\delta_1}$ to describe the main stability results all at once, regardless of whether the damping parameter $\alpha \geq 0$ vanishes or not.

We multiply $\Lam^{2s} u$, $\Lam^{2s} v$, and $\Lam^{2s} \theta$ on the both sides of $\eqref{E11}_1$-$\eqref{E11}_3$ respectively and integrate over $\mathbb{R}^2$. Then, we do similar process with $\Lam^{2\delta_1} u$, $v$, and $\te$ on $\eqref{E11}_1$-$\eqref{E11}_3$. By summing these two results and using $\lomu \leq \mu$, we have
\begin{equation}\label{pre_H^s_estimate}
\begin{split}
&\frac{1}{2} \frac{\ud}{\dt}
\left( \|\Lambda^s u\|_{L^2}^2
+ \|\Lambda^s v\|_{L^2}^2
+ \|\Lambda^s \te\|_{L^2}^2 
+ \|\Lambda^{\delta_1} u\|_{L^2}^2
+ \| v\|_{L^2}^2
+\| \te \|_{L^2}^2 \right) \\
&\quad + \left(\alpha \|\Lambda^s u\|_{L^2}^2   
+ \lomu\|\Lambda^s \nabla u\|_{L^2}^2 
+ \beta ||\Lambda^s v\|_{L^2}^2 
+ \alpha \|\Lambda^{\delta_1} u\|_{L^2}^2 
+ \lomu\|\Lambda^{\delta_1} \nabla u\|_{L^2}^2
+ \beta \| v \|_{L^2}^2\right) \\
&\leq - \int_{\bR^2} \Lambda^s ((u \cdot \nabla) u) \cdot \Lambda^s u
+ \int_{\bR^2} \Lambda^s (v \otimes  v) : \Lambda^s \nabla u 
- \int_{\bR^2} \Lambda^s ((u \cdot \nabla) v) \cdot \Lambda^s v
- \int_{\bR^2} \Lambda^s ((v \cdot \nabla) u) \cdot \Lambda^s v\\
&\quad - \int_{\bR^2} \Lambda^{ \delta_1} ((u \cdot \nabla) u) \cdot \Lambda^{ \delta_1} u
+ \int_{\bR^2} \Lambda^{ \delta_1} (v \otimes  v) : \Lambda^{ \delta_1}  \nabla u
- \int_{\bR^2} (v \cdot \nabla ) u \cdot v  - \int_{\bR^2} \Lambda^s (u \cdot \nabla \te) \Lambda^s \te \\
&\quad -\int_{\bR^2}\left(\Lam^s (\mu(\te)\nabla u) - \mu(\te)\Lam^s \nabla u \right):\Lam^s \nabla u
-\int_{\bR^2}\left(\Lam^{\delta_1} (\mu(\te)\nabla u) - \mu(\te)\Lam^{\delta_1} \nabla u \right):\Lambda^{\delta_1}\nabla u.
\end{split}
\end{equation}

\vspace{0.8em}

One difficulty in obtaining the stability result is the terms containing $\|\Lambda^s \te\|_{L^2(\bR^2)}$. 
To control these, we consider the \textit{cross term} with respect to $v, \theta$ as follows.

We multiply $\Lam^{2(s-1)} \nabla \theta$ on the second equation of \eqref{E11} and $\Lam^{2(s-1)} \dv v$ on the third equation of \eqref{E11}. Then, we integrate over $\mathbb{R}^2$ to obtain
\begin{align*}
\int_{\bR^2} v_t \Lam^{2(s-1)} \nabla\theta
+ \int_{\bR^2} (u \cdot \nabla)v \cdot\Lam^{2(s-1)} \nabla\theta
+ \int_{\bR^2} (v \cdot \nabla)u \cdot\Lam^{2(s-1)} \nabla\theta
+ \beta \int_{\bR^2} v \cdot\Lam^{2(s-1)} \nabla\theta
= \int_{\bR^2} | \Lam^{s} \theta|^2,& \\
\int_{\bR^2} \Lam^{2(s-1)} \theta_t \dv v
+ \int_{\bR^2} \Lam^{2(s-1)} (u \cdot \nabla \theta) \dv v 
=  \int_{\bR^2} | \Lam^{s-1} \dv v|^2.&
\end{align*}
By subtracting these two equations, we deduce that
\begin{equation}
\label{E31a}
\begin{split}
\int_{\bR^2} | \Lam^s \theta|^2
&- \beta \int_{\bR^2} \Lam^{s-1} v\cdot \Lam^{s-1}\nabla \theta
- \int_{\bR^2} | \Lam^{s-1}\dv v|^2 
- \frac{\ud}{\dt} \int_{\bR^2} \Lam^{s-1} v \cdot \Lam^{s-1}\nabla \theta \\
&=
\int_{\bR^2} \Lam^{s-1} (u \cdot \nabla)v \cdot \Lam^{s-1}\nabla \theta
+ \int_{\bR^2} \Lam^{s-1} (v \cdot \nabla)u \cdot \Lam^{s-1}\nabla \theta  - \int_{\bR^2} \Lam^{s-1} ( u \cdot \nabla \theta) \cdot \Lam^{s-1}\dv v.\\
\end{split}
\end{equation}
By Young's inequality, this leads to
\begin{equation}
\label{E32}
\begin{split}
\frac{1}{2}\|\Lambda^s\te\|_{L^2}^2& -\lt(1+\frac{\be^2}{2}\rt)\|v\|_{H^s}^2 
- \frac{\ud}{\dt} \int_{\bR^2} \Lam^{s-1} v \cdot \Lam^{s-1}\nabla \theta \\
&\leq
\int_{\bR^2} \Lam^{s-1} (u \cdot \nabla)v \cdot \Lam^{s-1}\nabla \theta
+ \int_{\bR^2} \Lam^{s-1} (v \cdot \nabla)u \cdot \Lam^{s-1}\nabla \theta  - \int_{\bR^2} \Lam^{s-1} ( u \cdot \nabla \theta) \cdot \Lam^{s-1}\dv v.\\
\end{split}
\end{equation}

Similarly, we take $s=1$ and apply the integration by parts to obtain 
\begin{equation}
\label{E35}
\begin{split}
\frac{1}{2}\|\nabla\te\|_{L^2}^2& -\lt(1+\frac{\be^2}{2}\rt)\|v\|_{H^1}^2 
- \frac{\ud}{\dt} \int_{\bR^2}  v \cdot \nabla \theta \\
&\leq
\int_{\bR^2}  (u \cdot \nabla)v \cdot \nabla \theta
+ \int_{\bR^2}  (v \cdot \nabla)u \cdot \nabla \theta  - \int_{\bR^2}  ( u \cdot \nabla \theta) \cdot \dv v = 2\int_{\bR^2}  (v \cdot \nabla)u \cdot \nabla \theta .\\
\end{split}
\end{equation}

Finally, we set $\eta$ to satisfy $0<\eta \leq \beta/(4+2\beta^2)$, which is always less than $\frac{1}{4}$.
We multiply \eqref{E32} and \eqref{E35} by a sufficiently small constant $\eta > 0$ and add it to \eqref{pre_H^s_estimate}, which yields
\begin{equation}
\label{E12}
\begin{split}
&\frac{1}{2} \frac{\ud}{\dt}
\left( \|\Lambda^s u\|_{L^2}^2
+ \|\Lambda^s v\|_{L^2}^2
+ \|\Lambda^s \te\|_{L^2}^2 
+ \|\Lambda^{\delta_1} u\|_{L^2}^2
+ \| v\|_{L^2}^2
+\| \te \|_{L^2}^2
-\eta  \int_{\bR^2} \Lam^{s-1} v \cdot \Lam^{s-1}\nabla \theta
-\eta\int_{\bR^2}  v \cdot \nabla \theta
\right) \\
&\quad + \left(\alpha \|\Lambda^s u\|_{L^2}^2
+ \lomu\|\Lambda^s \nabla u\|_{L^2}^2 
+ \frac{\beta}{2} ||\Lambda^s v\|_{L^2}^2 
+ \alpha \|\Lambda^{\delta_1} u\|_{L^2}^2 
+ \lomu\|\Lambda^{\delta_1} \nabla u\|_{L^2}^2
+ \frac{\beta}{2} \| v \|_{L^2}^2\right)
+ \frac{\eta}{2} \|\nabla \te\|_{H^{s-1}}^2 \\
&\leq - \int_{\bR^2} \Lambda^s ((u \cdot \nabla) u) \cdot \Lambda^s u
+ \int_{\bR^2} \Lambda^s (v \otimes  v) : \Lambda^s \nabla u 
- \int_{\bR^2} \Lambda^s ((u \cdot \nabla) v) \cdot \Lambda^s v
- \int_{\bR^2} \Lambda^s ((v \cdot \nabla) u) \cdot \Lambda^s v\\
&\quad - \int_{\bR^2} \Lambda^{ \delta_1} ((u \cdot \nabla) u) \cdot \Lambda^{ \delta_1} u
+ \int_{\bR^2} \Lambda^{ \delta_1} (v \otimes  v) : \Lambda^{ \delta_1}  \nabla u
- \int_{\bR^2} (v \cdot \nabla ) u \cdot v  - \int_{\bR^2} \Lambda^s (u \cdot \nabla \te) \Lambda^s \te \\
&\quad -\int_{\bR^2}\left(\Lam^s (\mu(\te)\nabla u) - \mu(\te)\Lam^s \nabla u \right):\Lam^s \nabla u
-\int_{\bR^2}\left(\Lam^{\delta_1} (\mu(\te)\nabla u) - \mu(\te)\Lam^{\delta_1} \nabla u \right):\Lambda^{\delta_1}\nabla u \\
&\quad +
\eta\int_{\bR^2} \Lam^{s-1} (u \cdot \nabla)v \cdot \Lam^{s-1}\nabla \theta 
+ \eta\int_{\bR^2} \Lam^{s-1} (v \cdot \nabla)u \cdot \Lam^{s-1}\nabla \theta 
- \eta\int_{\bR^2} \Lam^{s-1} ( u \cdot \nabla \theta) \cdot \Lam^{s-1}\dv v \\
& \quad+ 2\eta\int_{\bR^2}  (v \cdot \nabla)u \cdot \nabla \theta\\
&=: \sum_{k=1}^{14}\cI_k.\end{split}
\end{equation}

\begin{lemma}
\label{L31}
Let $\alpha \geq 0$ and $s > 1$. 
Let $\delta_1$ and $\lambda$ as defined in \eqref{delta_1} and in \eqref{lambda}, respectively.
Let $(u,v,\te)$ be a smooth solution of \eqref{E11} on $[0,T] \times \bR^2$. For each time $t>0$, we define $A_s(t)$ and $B_s(t)$ by
\[
A_s(t) := 
\left( 
\|\Lambda^s u(t)\|_{L^2}^2
+\|\Lambda^{\delta_1} u(t)\|_{L^2}^2
+\| v(t)\|^2_{H^s}
+\|\te(t)\|^2_{H^s}
-\eta \int_{\bR^2} \Lam^{s-1} v \cdot \Lam^{s-1}\nabla \theta
-\eta\int_{\bR^2}  v \cdot \nabla \theta
\right)^{\frac{1}{2}} 
\]
and
\begin{equation}
\label{B_s(t)}
B_s(t) := \lambda\left( \|\nabla u(t)\|_{H^s}^2 
+ \| v(t) \|^2_{H^s} + \|\nabla \te(t)\|_{H^{s-1}}^2 \right)^{\frac{1}{2}}.
\end{equation}
Then, there exists a constant $C>0$ which depends on $\al, \be, s,$ and $\lomu$ such that
\begin{equation}
\label{ingredi_direct-2}
\frac{1}{2}\frac{\ud}{\dt}A_s^2 + B_s^2 \leq C\,(A_s + A_s^{s+1})\,B_s^2. 
\end{equation}
\end{lemma}
\begin{remark}
\label{rmk315}
Since $\eta$ is chosen to be less than $\frac{1}{4}$, we have
\[
       \frac{3}{4}A_s(t)^2 \leq \left( 
\|\Lambda^s u(t)\|_{L^2}^2
+\|\Lambda^{\delta_1} u(t)\|_{L^2}^2
+\| v(t)\|_{H^s}^2
+\|\te(t)\|_{H^s}^2
\right) \leq \frac{5}{4}A_s(t)^2.
\]
Moreover, we have
\[
    B_s(t)^2 \leq \left(\alpha \|\Lambda^s u\|_{L^2}^2
+ \lomu\|\Lambda^s \nabla u\|_{L^2}^2 
+ \frac{\beta}{2} ||\Lambda^s v\|_{L^2}^2 
+ \alpha \|\Lambda^{\delta_1} u\|_{L^2}^2 
+ \lomu\|\Lambda^{\delta_1} \nabla u\|_{L^2}^2
+ \frac{\beta}{2} \| v \|_{L^2}^2\right)
+ \frac{\eta}{2} \|\nabla \te\|_{H^{s-1}}^2,
\]
for the both cases $\alpha>0$ or $\alpha=0$. This simplicity motivates the definition of $\lambda$.
\end{remark}
\begin{proof}[Proof of Lemma~\ref{L31}]
Due to \eqref{E12}, it only remains to bound nonlinear terms $\sum_{k=1}^{13}\cI_k$.
We consider two cases, either $ \alpha = 0$ or $ \alpha > 0$. This simplicity is the motivation for the definition of $\lambda$.

Let us consider the undamped case, i.e., $ \alpha = 0$. In this case, we have $\delta_1 = 0$. By using integration by parts, we can easily obtain
\[
\cI_5 = \int_{\bR^2} (u \cdot \nabla u) \cdot u = 0,\quad \text{and} \quad 
\cI_6 + \cI_7 = \int_{\bR^2} (v \otimes  v) : \nabla u
- \int_{\bR^2} (v \cdot \nabla ) u \cdot v
= 0.
\]
Moreover, it is clear that $\cI_{10}=0$.

The estimates for the rest terms follow from a combination of \eqref{E14}, Lemma~\ref{L23}, Lemma~\ref{L24}, and the divergence-free condition $\dv u = 0$. Thus, we obtain
\begin{align*}
\cI_1 
&=\lt|\, \int_{\bR^2} (\Lambda^s ((u \cdot \nabla) u) -(u\cdot \nabla) \Lambda^s u) \cdot \Lambda^s u \,\rt|
\lesssim
\norm{ \nabla u}_{L^2}
\norm{\Lambda^s u}_{L^4}^2
\lesssim
\norm{ \nabla u}_{L^2}
\norm{\Lambda^{s+\frac{1}{2}} u}_{L^2}^2\\
&\lesssim
\norm{ \nabla u}_{L^2}
(\norm{\Lam^{s+1} u}_{L^2}^{\frac{1}{2}}
\norm{\Lambda^s u}_{L^2}^{\frac{1}{2}})^2
\lesssim
\norm{ \nabla u}_{L^2}
\norm{\Lam^{s+1} u}_{L^2}
\norm{\Lambda^s u}_{L^2} \lesssim A_s B_s B_s, \\
\cI_2
&=\lt|\,\int_{\bR^2} \Lambda^s (v \otimes  v) : \Lambda^s \nabla u \,\rt|
\lesssim \norm{ \Lam^s v}_{L^2}
\norm{v}_{L^\infty}
\norm{ \Lam^{s+1} u}_{L^2}
\lesssim
\norm{ \Lam^s v}_{L^2}
\norm{v}_{H^s}
\norm{ \Lam^{s+1} u}_{L^2}
\lesssim A_s B_s B_s
, \\
\cI_3
&=\lt| \, \int_{\bR^2} (\Lambda^s ((u \cdot \nabla) v) -(u\cdot \nabla) \Lambda^s v) \cdot \Lambda^s v \,\rt|\\
&\lesssim
\norm{ \nabla u}_{L^\infty}
\norm{ \Lam^s v}_{L^2}^2
+
\begin{cases}
 \norm{ \Lam^s u}_{L^\frac{2}{s-1}}
\norm{ \nabla v}_{L^\frac{2}{2-s}}
\norm{ \Lam^s v}_{L^2} &\quad \text{if} \ 1 < s < 2, \\
\norm{ \Lam^s u}_{L^4}
\norm{ \nabla v}_{L^4}
\norm{ \Lam^s v}_{L^2} &\quad \text{if} \ s \ge 2,
\end{cases}
\\
&\lesssim
\|\nabla u\|_{H^s}\|\Lambda^s v\|_{L^2}^2 +
\begin{cases}
\norm{ \Lam^{2} u}_{L^2}
\norm{ \Lam^s v}_{L^2}^2 &\quad \text{if} \ 1 < s < 2, \\
\norm{ \Lam^{s+\frac{1}{2}} u}_{L^2}
\norm{ \Lam^{\frac{3}{2}} v}_{L^2}
\norm{ \Lam^s v}_{L^2} &\quad \text{if} \ s \ge 2,
\end{cases}
\\
& \lesssim
\|\nabla u\|_{H^s}\|v\|_{H^s}\|\Lambda^s v\|_{L^2}
\lesssim B_s A_s B_s,
\\
\cI_4
&= \lt| \, \int_{\bR^2} \Lambda^s ((v \cdot \nabla) u) \cdot \Lambda^s v \,\rt| \\
&\lesssim \norm{ \Lam^s v}_{L^2}^2
\norm{ \nabla u}_{L^\infty}
+ \norm{v}_{L^\infty}
\norm{ \Lam^{s+1} u}_{L^2}
\norm{ \Lam^s v}_{L^2} \\
&\lesssim  \norm{ \Lam^s v}_{L^2}^2
\|\nabla u\|_{H^s}
+ \norm{ v}_{H^s}
\norm{ \Lam^{s+1} u}_{L^2}
\norm{ \Lam^s v}_{L^2}
\lesssim A_s B_s B_s,
\\
\cI_8 &=
\left|\int_{\bR^2} (\Lambda^s (u \cdot \nabla \te)- (u\cdot\nabla) \Lam^s \te)\, \Lambda^s \te  \, \right|\\
&\lesssim \|\nabla u\|_{L^{\infty}}
\|\Lambda^s \te\|^2_{L^2} +
\begin{cases}
    \|\Lambda^s u\|_{L^{\frac{2}{s-1}}}
    \|\nabla \theta\|_{L^{\frac{2}{2-s}}}
    \|\Lambda^s \te\|_{L^2} 
    &\quad\text{if }1 < s < 2, \\
    \|\Lambda^s u\|_{L^4}
    \|\nabla \theta\|_{L^4}
    \|\Lambda^s \te\|_{L^2} 
    & \quad \text{if }s \geq 2, \\   
    \end{cases}\\
&\lesssim  
\|\nabla u\|_{H^s}
\|\Lambda^s \te\|_{L^2}^2 
+\begin{cases}
\norm{ \Lam^{2} u}_{L^2}
\norm{ \Lam^s \te}_{L^2}^2 
&\, \text{ if} \ 1 < s < 2, \\
\norm{ \Lam^{s+\frac{1}{2}} u}_{L^2}
\norm{ \Lam^\frac{3}{2} \te}_{L^2}
\|\Lambda^s \te\|_{L^2} 
&\,\text{ if} \ s \ge 2,
\end{cases}
\\
& \lesssim \|\te\|_{H^s}\|\Lambda^s \te\|_{L^2}\|\nabla u\|_{H^s}
\lesssim  A_s B_s B_s.
\end{align*}
Moreover, by using Lemma~\ref{L24}, \eqref{E70}, and \eqref{E71}, it follows that
\begin{align*}
\cI_{9}
&=\left|\int_{\bR^2} 
 \left(\Lam^s (\mu(\te)\nabla u) - \mu(\te)\Lam^s \nabla u \right):\Lam^s \nabla u \, \right|\\
&=\left|\int_{\bR^2} 
 \left(\Lam^s ((\mu(\te)-\mu(0))\nabla u) - (\mu(\te)-\mu(0))\Lam^s \nabla u \right):\Lam^s \nabla u \, \right| \\ 
&\lesssim
\|\Lambda^s (\mu(\te)-\mu(0))\|_{L^{2}}
\|\nabla u \|_{L^{\infty}}
\|\Lambda^{s+1} u\|_{L^2} 
+ \begin{cases}
\|\Lambda^s u\|_{L^{\frac{2}{s-1}}}\|\nabla \mu(\theta)\|_{L^{\frac{2}{2-s}}}\|\Lambda^{s+1} u\|_{L^2} &\,\text{ if } 1 < s < 2,\\
\|\Lambda^s u\|_{L^4}\|\nabla \mu(\theta)\|_{L^4}\|\Lambda^{s+1} u\|_{L^2} & \, \text{ if }s \geq 2, \\   
\end{cases}\\
&\lesssim (1+\|\nabla \te\|_{L^2}^{\lceil s-1 \rceil})
\|\Lambda^s \te\|_{L^2} \| \nabla u \|_{H^s}\|\Lambda^{s+1} u\|_{L^2} + \| \te\|_{H^s}\|\nabla u\|_{H^s}\|\Lambda^{s+1} u\|_{L^2}\\
&\lesssim (
\|\te\|_{H^s}^{s+1} 
+ \|\te\|_{H^s})\|\nabla u\|_{H^s} 
\|\Lambda^{s+1} u\|_{L^2}
\lesssim 
A_s^{s+1}
B_s^2 
+ A_s B_s^2.
\end{align*}
The last three terms can be estimated by
\begin{align*}
\cI_{11}
=\lt|\, \int_{\bR^2} \Lam^{s-1} (u \cdot \nabla)v \cdot \Lam^{s-1}\nabla \theta \,\rt|
&\lesssim
\norm{u}_{L^\infty}
\norm{ \Lam^s v}_{L^2} 
\norm{ \Lam^s \theta}_{L^2}
+ 
\begin{cases}
\norm{ \Lam^{s-1} u}_{L^\frac{2}{s-1}}
\norm{ \nabla v}_{L^\frac{2}{2-s}}
\norm{ \Lam^s \theta}_{L^2}
&\, \text{if} \ 1 < s < 2, \\
\norm{ \Lam^{s-1} u}_{L^4}
\norm{ \nabla v}_{L^4}
\norm{ \Lam^s \theta}_{L^2}
&\, \text{if} \ s \ge 2,
\end{cases}
\\
&\lesssim
\norm{u}_{H^s}
\norm{ \Lam^{s} v}_{L^2}
\norm{ \Lam^s \theta}_{L^2}
+
\begin{cases}
\norm{ \nabla u}_{L^2}
\norm{ \Lam^s v}_{L^2}
\norm{ \Lam^s \theta}_{L^2}
&\, \text{if} \ 1 < s < 2, \\
\norm{ \Lam^{s-\frac{1}{2}} u}_{L^2}
\norm{ \Lam^{\frac{3}{2}} v}_{L^2}
\norm{ \Lam^s \theta}_{L^2}
&\, \text{if} \ s \ge 2,
\end{cases}
\\
&\lesssim
\norm{u}_{H^s}
\norm{ v}_{H^s}
\norm{ \Lam^s \theta}_{L^2}
\lesssim A_s B_s B_s,
\\
\cI_{12}
=\lt|\,\int_{\bR^2} \Lam^{s-1} (v \cdot \nabla)u \cdot \Lam^{s-1}\nabla \theta  \,\rt|
&\lesssim \lt(\norm{ \Lam^{s-1} v}_{L^2}
\norm{ \nabla u}_{L^\infty}
+ \norm{v}_{L^\infty}
\norm{ \Lam^s u}_{L^2}\rt)
\norm{ \Lam^s \theta}_{L^2} \\
&
\lesssim
\norm{v}_{H^s}
\norm{ \nabla u}_{H^s}
\norm{ \Lam^s \theta}_{L^2}
\lesssim A_s B_s B_s,
\end{align*}
and
\begin{align*}
\cI_{13}
=\lt|\, \int_{\bR^2} \Lam^{s-1} ( u \cdot \nabla \theta) \cdot \Lam^{s-1}\dv v \,\rt|
&\lesssim
\norm{u}_{L^\infty}
\norm{ \Lam^s \theta}_{L^2}
\norm{ \Lam^s v}_{L^2} 
+
\begin{cases}
\norm{ \Lam^{s-1} u}_{L^\frac{2}{s-1}}
\norm{ \nabla \theta}_{L^\frac{2}{2-s}}
\norm{ \Lam^s v}_{L^2}
&\,\text{if} \ 1 < s < 2, \\
\norm{ \Lam^{s-1} u}_{L^4}
\norm{ \nabla \theta}_{L^4}
\norm{ \Lam^s v}_{L^2}
&\, \text{if} \ s \ge 2,
\end{cases}\\
&\lesssim
\norm{u}_{H^s}
\norm{ \Lam^s \theta}_{L^2}
\norm{ \Lam^s v}_{L^2}
+
\begin{cases}
\norm{ \nabla u}_{L^2}
\norm{ \Lam^s \theta}_{L^2}
\norm{ \Lam^s v}_{L^2}
&\, \text{if} \ 1 < s < 2, \\
\norm{ \Lam^{s-\frac{1}{2}} u}_{L^2}
\norm{ \Lam^{\frac{3}{2}} \theta}_{L^2}
\norm{ \Lam^s v}_{L^2}
&\, \text{if} \ s \ge 2,
\end{cases}
\\
&\lesssim
\norm{u}_{H^s}
\norm{ \Lam^s \theta}_{L^2}
\norm{ \Lam^s v}_{L^2}
+
\norm{\theta}_{H^s}
\norm{\nabla u}_{H^s}
\norm{ \Lam^s v}_{L^2}
\lesssim A_s B_s B_s.
\end{align*}
Finally, we have
\[
\cI_{14}= 
\left|\, 2\eta\int_{\bR^2}  (v \cdot \nabla)u \cdot \nabla \theta\, \right|
\lesssim 
\norm{v}_{L^{\infty}}
\norm{ \nabla u}_{L^2}
\norm{ \nabla \te}_{L^2}
\lesssim 
\norm{v}_{H^s}
\norm{ \nabla u}_{L^2}
\norm{ \nabla \te}_{L^2}
\lesssim B_s A_s B_s.
\]

Combining all these estimates, we conclude that all the terms 
$\cI_k$ are bounded by $C(A_sB_s^2 + A_s^{s+1} B_s^2)$, where $C$ is a constant depending only on $\be>0$, $\lomu >0$, and $s>1$. This completes the estimate \eqref{ingredi_direct-2} in the undamped case.

Next, let us consider the damped case, i.e. $\alpha > 0$. In this case, we have $\delta_1 = 1$. For the terms $\cI_1-\cI_4$, $\cI_8$, $\cI_9$, $\cI_{12}$, and $\cI_{14}$, the same estimate as the undamped case can be used.
By using \eqref{E14}, Lemma~\ref{L23}, and $\dv u=0$, we see that
\begin{align*}
\cI_5
&=\lt|\,\int_{\bR^2} (\Lambda ((u \cdot \nabla) u)-(u\cdot \nabla)\Lambda u) \cdot \Lambda u \,\rt|
\\
&\lesssim \norm{ \nabla u}_{L^2}
\norm{ \nabla u}_{L^4}
\norm{ \nabla u}_{L^4}  \lesssim \norm{ \nabla u}_{L^2}
(\norm{ \nabla u}_{L^2}^{\frac{1}{2}}
\norm{ \Lam^2 u}_{L^2}^{\frac{1}{2}})^2 \lesssim \norm{ \nabla u}_{L^2}^2
\norm{ \nabla u}_{H^s} \lesssim A_s B_s B_s, \\
\cI_6
&=\lt|\, \int_{\bR^2} \Lambda (v \otimes  v) : \Lambda  \nabla u
\, \rt|
\lesssim \norm{ \nabla v}_{L^2}
\norm{ v}_{L^\infty}
\norm{ \Lam^2 u}_{L^2}
\lesssim \norm{ \nabla v}_{L^2}
\norm{v}_{H^s} 
\norm{ \nabla u}_{H^s}
\lesssim A_s B_s B_s
, \\
\cI_7
&= \lt|\, \int_{\bR^2} (v \cdot \nabla ) u \cdot v  \,\rt|
\lesssim \norm{v}_{L^\infty}
\norm{ \nabla u}_{L^2}
\norm{v}_{L^2}
\lesssim \norm{v}_{H^s} 
\norm{ \nabla u}_{L^2}
\norm{v}_{L^2}
\lesssim A_s B_s B_s.
\end{align*}
The identity $\nabla \mu(\te) = \mu'(\te)\nabla \te$ yields that~
\[
\cI_{10}
= \left|\int_{\bR^2} 
 \left(\Lam (\mu(\te)\nabla u) - \mu(\te)\Lam  \nabla u \right):\Lam \nabla u \, \right|
\lesssim  
\|\nabla \mu(\te)\|_{L^{2}}
\|\nabla u \|_{L^{\infty}}
\|\Lambda^{2} u\|_{L^2} 
\lesssim  \|\nabla \te\|_{L^2}\|\nabla u\|_{H^s}^2 
\lesssim A_s B_s^2.
\]
Compared to the undamped case, however, the $\cI_{11}$ and $\cI_{13}$ terms require a technically different approach. Since we no longer have an appropriate estimate for $\|u\|_{L^{\infty}}$ in the damped case, we shall use cancellation properties and integration by parts.

We consider the following identities, where we used the Einstein summation convention,
\begin{align*}
\cI_{11}
&= \int_{\bR^2} \left( \Lam^{s-1} (u_i \p_i v_j) 
- u_i \p_i \Lam^{s-1} v_j \right) \p_j \Lam^{s-1} \theta
+ \int_{\bR^2} u_i \p_i \Lam^{s-1} v_j \p_j \Lam^{s-1} \theta, \\
\cI_{13}
&= -\int_{\bR^2} \left( \Lam^{s-1} (u_i \p_i \theta)
- u_i \p_i \Lam^{s-1} \theta \right) \Lam^{s-1} \p_j v_j
- \int_{\bR^2} u_i \p_i \Lam^{s-1} \theta \Lam^{s-1} \p_j v_j.
\end{align*}
Note that one can derive the following identity by using integration by parts twice and $\dv u = 0$ condition of $u$,
\[
- \int_{\bR^2} u_i \p_i \Lam^{s-1} v_j \p_j \Lam^{s-1} \theta
= \int_{\bR^2} u_i \Lam^{s-1} v_j  \p_i \left(\p_j \Lam^{s-1} \theta\right)
= - \int_{\bR^2}  \p_i \Lam^{s-1}\theta \left( \p_j u_i \Lam^{s-1} v_j
+ u_i \Lam^{s-1} \p_j v_j \right).
\]
Therefore, we deduce that
\begin{equation}
\label{Power1}
\begin{split}
\cI_{11} + \cI_{13}
&=  \int_{\bR^2} \left( \Lam^{s-1} (u_i \p_i v_j) 
- u_i \p_i \Lam^{s-1} v_j \right) \p_j \Lam^{s-1} \theta
-  \int_{\bR^2} \left( \Lam^{s-1} (u_i \p_i \theta)
- u_i \p_i \Lam^{s-1} \theta \right) \Lam^{s-1} \p_j v_j \\
&\quad + \int_{\bR^2}  (\p_i \Lam^{s-1} \theta) (\p_j u_i)(\Lam^{s-1} v_j)
=: N_1 + N_2 + N_3.
\end{split}
\end{equation}
This representation makes us compute these terms as follows:
\begin{equation}
\label{E27}
\begin{split}
N_1 
&\lesssim \norm{ \nabla u}_{L^{\infty}}
\norm{ \Lam^{s-1} v}_{L^2}
\norm{ \Lam^s \theta}_{L^2}
+
\begin{cases}
\norm{ \Lam^{s-1} u}_{L^{\frac{2}{s-1}}}
\norm{ \nabla v}_{L^{\frac{2}{2-s}}}
\norm{ \Lam^s \theta}_{L^2} 
&\quad \text{if }1 < s < 2, \\
\norm{ \Lam^{s-1} u}_{L^4}
\norm{ \nabla v}_{L^4}
\norm{ \Lam^s \theta}_{L^2} 
&\quad \text{if } s \geq 2,
\end{cases} \\
& \lesssim 
\norm{ \nabla u}_{H^s}
\norm{ v}_{H^s}
\norm{ \Lam^s \theta}_{L^2}
+
\begin{cases}
\norm{ \nabla u}_{L^2}
\norm{ \Lambda^s v}_{L^2}
\norm{ \Lam^s \theta}_{L^2}
&\quad \text{if }1 < s < 2, 
\\
\norm{ \Lam^{s-\frac{1}{2}} u}_{L^2}
\norm{ \Lambda^{\frac{3}{2}}v}_{L^2}
\norm{ \Lam^s \theta}_{L^2}
&\quad \text{if } s \geq 2,
\end{cases}
\\
& \lesssim
\norm{ \nabla u}_{H^s}
\norm{ v}_{H^s}
\norm{ \Lam^s \theta}_{L^2}.
\end{split}
\end{equation}
A similar computation reveals that
\begin{equation}
\label{E29-210}
\begin{split}
N_2 &\lesssim \norm{ \nabla u}_{L^{\infty}}
\norm{ \Lam^{s-1} \te}_{L^2}
\norm{ \Lam^s v}_{L^2}
+
\begin{cases}
\norm{ \Lam^{s-1} u}_{L^{\frac{2}{s-1}}}
\norm{ \nabla \te}_{L^{\frac{2}{2-s}}}
\norm{ \Lam^s v}_{L^2}
&\quad \text{if }1 < s < 2, 
\\
\norm{ \Lam^{s-1} u}_{L^4}
\norm{ \nabla \te}_{L^4}
\norm{ \Lam^s v}_{L^2}
&\quad \text{if } s \geq 2,
\end{cases} \\
&\lesssim \norm{ \nabla u}_{H^s}
\norm{ \te}_{H^s}
\norm{ \Lam^s v}_{L^2},
\end{split}
\end{equation}
and
\begin{equation}
\label{E211}
    N_3 \lesssim \norm{ \Lam^s \theta}_{L^2}\norm{ \nabla u}_{L^{\infty}}
\norm{ \Lam^{s-1} v}_{L^2}
 \lesssim \norm{ \Lam^s \theta}_{L^2}\norm{ \nabla u}_{H^s}
\norm{ v}_{H^s}.
\end{equation}
We combine \eqref{E27}-\eqref{E211} to conclude that
\[
    \cI_{11}+\cI_{13} \lesssim \|\te\|_{H^s} \|\nabla u \|_{H^s}\|v\|_{H^s}
    \lesssim A_s B_s B_s.
\]
Thus, we have established \eqref{ingredi_direct-2} for the damped case too, and the proof is complete.
\end{proof}

\subsection{Temporal decay estimates}\label{sec_3.2}
Next, we establish a lemma concerning the temporal decay estimates of $\dot{H}^{s-1}$ norm.
To simplify the computation, we assume that $s \in(1,2)$ in this subsection. In addition, we assume that the solutions are smooth and establish an a priori estimate as in the previous section.
We consider the temperature $\theta$ in the absence of thermal diffusion,
and thus the damping term $\norm{\Lam^s \theta}_{L^2}^2$ is absent from the energy estimates.
To handle this difficulty, we employ the operators $\Lambda^{2s}$ and $\Lambda^{2(s-1)}$, along with the cross term.

Let us multiply equations \eqref{E11}$_1$–\eqref{E11}$_3$ by $\Lam^{2s} u$, $\Lam^{2s} v$, and $\Lam^{2s} \theta$, respectively and integrate over $\mathbb{R}^2$. We then perform a similar procedure using $\Lam^{2(s-1)} u$, $\Lam^{2(s-1)}v$, and $\Lam^{2(s-1)}\te$. Summing the resulting identities and using the bound $\lomu \leq \mu$, we obtain
\begin{align*}
&\frac{1}{2} \frac{\ud}{\dt}
\left( \|\Lambda^s u\|_{L^2}^2
+ \|\Lambda^s v\|_{L^2}^2
+ \|\Lambda^s \te\|_{L^2}^2 
+ \|\Lambda^{s-1} u\|_{L^2}^2
+ \|\Lambda^{s-1} v\|_{L^2}^2
+\|\Lambda^{s-1} \te \|_{L^2}^2 \right) \\
&\quad + \left(\alpha \|\Lambda^s u\|_{L^2}^2
+ \lomu\|\Lambda^s \nabla u\|_{L^2}^2 
+ \beta ||\Lambda^s v\|_{L^2}^2 
+ \alpha \|\Lambda^{s-1} u\|_{L^2}^2 
+ \lomu\|\Lambda^{s-1} \nabla u\|_{L^2}^2
+ \beta \|\Lambda^{s-1} v \|_{L^2}^2 \right) \\
&\leq - \int_{\bR^2} \Lambda^s ((u \cdot \nabla) u) \cdot \Lambda^s u
+ \int_{\bR^2} \Lambda^s (v \otimes  v) : \Lambda^s \nabla u 
- \int_{\bR^2} \Lambda^s ((u \cdot \nabla) v) \cdot \Lambda^s v
- \int_{\bR^2} \Lambda^s ((v \cdot \nabla) u) \cdot \Lambda^s v\\
&\quad - \int_{\bR^2} \Lambda^{s-1} ((u \cdot \nabla) u) \cdot \Lambda^{s-1} u
+ \int_{\bR^2} \Lambda^{s-1} (v \otimes  v) : \Lambda^{s-1} \nabla u 
- \int_{\bR^2} \Lambda^{s-1} ((u \cdot \nabla) v) \cdot \Lambda^{s-1} v\\
&
\quad - \int_{\bR^2} \Lambda^{s-1} ((v \cdot \nabla) u) \cdot \Lambda^{s-1} v  - \int_{\bR^2} \Lambda^s (u \cdot \nabla \te) \Lambda^s \te -\int_{\bR^2}\left(\Lam^s (\mu(\te)\nabla u) - \mu(\te)\Lam^s \nabla u \right):\Lam^s \nabla u\\
&
\quad - \int_{\bR^2} \Lambda^{s-1} (u \cdot \nabla \te) \Lambda^{s-1} \te -\int_{\bR^2}\left(\Lam^{s-1} (\mu(\te)\nabla u) - \mu(\te)\Lam^{s-1} \nabla u \right):\Lam^{s-1} \nabla u.\\
\end{align*}
To address a similar difficulty as before, we again introduce a cross term. Using the previously derived identity \eqref{E31a} together with Young’s inequality, we obtain the following estimate
\begin{align*}
\frac{1}{2}\int_{\bR^2} | \Lam^s \theta|^2
& 
 - \frac{\beta^2}{2} \int_{\bR^2} |\Lam^{s-1} v|^2
 - \int_{\bR^2} | \Lam^{s-1}\dv v|^2 -\frac{\ud}{\dt} \int_{\bR^2} \Lam^{s-1} v \cdot \Lam^{s-1}\nabla \theta \\
 & 
\leq \int_{\bR^2} \Lam^{s-1} (u \cdot \nabla)v \cdot \Lam^{s-1}\nabla \theta
+ \int_{\bR^2} \Lam^{s-1} (v \cdot \nabla)u \cdot \Lam^{s-1}\nabla \theta 
- \int_{\bR^2} \Lam^{s-1} ( u \cdot \nabla \theta) \cdot \Lam^{s-1}\dv v. 
\end{align*}
We then multiply sufficiently small constant $\kappa>0$ to \eqref{E31a} and add to \eqref{pre_H^s_estimate} to obtain
\begin{equation}
\label{H^s_temp}
\begin{split}
&\frac{1}{2} \frac{\ud}{\dt}
\left( \|\Lambda^s u\|_{L^2}^2
+ \|\Lambda^s v\|_{L^2}^2
+ \|\Lambda^s \te\|_{L^2}^2 
+ \|\Lambda^{s-1} u\|_{L^2}^2
+ \|\Lambda^{s-1} v\|_{L^2}^2
+\|\Lambda^{s-1} \te \|_{L^2}^2
-\kappa \int_{\bR^2}\Lambda^{s-1}v \cdot \Lambda^{s-1}\nabla \te
\,\right) \\
&\quad + \left(\alpha \|\Lambda^s u\|_{L^2}^2
+ \lomu\|\Lambda^s \nabla u\|_{L^2}^2 
+ \frac{\beta}{2} ||\Lambda^s v\|_{L^2}^2 
+ \alpha \|\Lambda^{s-1} u\|_{L^2}^2 
+ \lomu\|\Lambda^{s-1} \nabla u\|_{L^2}^2
+ \beta \|\Lambda^{s-1} v \|_{L^2}^2\right)
+ \frac{\kappa}{2}\|\Lambda^s \te\|_{L^2}^2
\\
&\leq - \int_{\bR^2} \Lambda^s ((u \cdot \nabla) u) \cdot \Lambda^s u
+ \int_{\bR^2} \Lambda^s (v \otimes  v) : \Lambda^s \nabla u 
- \int_{\bR^2} \Lambda^s ((u \cdot \nabla) v) \cdot \Lambda^s v
- \int_{\bR^2} \Lambda^s ((v \cdot \nabla) u) \cdot \Lambda^s v\\
&\quad - \int_{\bR^2} \Lambda^{s-1} ((u \cdot \nabla) u) \cdot \Lambda^{s-1} u
+ \int_{\bR^2} \Lambda^{s-1} (v \otimes  v) : \Lambda^{s-1} \nabla u 
- \int_{\bR^2} \Lambda^{s-1} ((u \cdot \nabla) v) \cdot \Lambda^{s-1} v\\
&
\quad - \int_{\bR^2} \Lambda^{s-1} ((v \cdot \nabla) u) \cdot \Lambda^{s-1} v  - \int_{\bR^2} \Lambda^s (u \cdot \nabla \te) \Lambda^s \te -\int_{\bR^2}\left(\Lam^s (\mu(\te)\nabla u) - \mu(\te)\Lam^s \nabla u \right):\Lam^s \nabla u\\
&
\quad - \int_{\bR^2} \Lambda^{s-1} (u \cdot \nabla \te) \Lambda^{s-1} \te -\int_{\bR^2}\left(\Lam^{s-1} (\mu(\te)\nabla u) - \mu(\te)\Lam^{s-1} \nabla u \right):\Lam^{s-1} \nabla u\\
&\quad+\kappa \int_{\bR^2} \Lam^{s-1} (u \cdot \nabla)v \cdot \Lam^{s-1}\nabla \theta
+\kappa \int_{\bR^2} \Lam^{s-1} (v \cdot \nabla)u \cdot \Lam^{s-1}\nabla \theta 
-\kappa \int_{\bR^2} \Lam^{s-1} ( u \cdot \nabla \theta) \cdot \Lam^{s-1}\dv v\\
&=: \sum_{k=1}^{15}\cI_k,
\end{split}
\end{equation}
where we chose $\kappa$ to satisfy $0<\kappa \leq \min(\frac{\beta}{2},\frac{1}{\beta}) \leq \frac{1}{\sqrt{2}}$. 

We may now state the following lemma. Recall the number $\delta_1$ in \eqref{delta_1}.
\begin{lemma}
 \label{L33}
Let $\alpha \geq 0$ and $s \in(1,2)$. Let $(u,v,\te)$ be a smooth solution of \eqref{E11} on $[0,\infty) \times \bR^2 $.  
For each time $t>0$, we define $X_s(t)$ and $Y_s(t)$ by
\begin{equation}
\label{X_s(t)}
X_s(t) := \left( \|\Lambda^s u\|_{L^2}^2
+ \|\Lambda^s v\|_{L^2}^2
+ \|\Lambda^s \te\|_{L^2}^2 
+ \|\Lambda^{s-1} u\|_{L^2}^2
+ \|\Lambda^{s-1} v\|_{L^2}^2
+\|\Lambda^{s-1} \te \|_{L^2}^2 
-\kappa \int_{\bR^2}\Lambda^{s-1}v \cdot \Lambda^{s-1}\nabla \te
\,\right)^{\frac{1}{2}}(t) 
\end{equation}
and
\begin{equation}
\label{Y_s(t)}
Y_s(t) := \left(
\|\Lambda^{s+1} u\|_{L^2}^2
+ \|\Lambda^{s}  u\|_{L^2}^2
+ \al\|\Lambda^{s-1} u\|_{L^2}^2 
+ \|\Lambda^s v\|_{L^2}^2 
+ \|\Lambda^{s-1} v \|_{L^2}^2
+ \|\Lambda^s \te\|_{L^2}^2
\right)^{\frac{1}{2}}(t). 
\end{equation}
There exists a constant $\epsilon_3>0$ such that the following a priori estimate holds: if $0<\ep<\ep_3$ and
\begin{equation}
\label{E54}
\|\Lambda^su(t)\|_{L^2}
+\|\Lambda^{\delta_1}u(t)\|_{L^2}
+\|v(t)\|_{H^s}
+\|\te(t)\|_{H^s} \leq 2\epsilon \quad \text{ for all } t\geq0, 
\end{equation}
then we have
\begin{equation}
\label{E33}
\frac{\ud}{\dt}\left(X_s(t)\right)^2 + \left(Y_s(t)\right)^2 \leq 0 \quad \text{for all }t\geq0.
\end{equation}
\end{lemma}
\begin{remark}
Notice that
\[
        X_s(t) \simeq \left( \|\Lambda^s u\|_{L^2}^2
+ \|\Lambda^s v\|_{L^2}^2
+ \|\Lambda^s \te\|_{L^2}^2 
+ \|\Lambda^{s-1} u\|_{L^2}^2
+ \|\Lambda^{s-1} v\|_{L^2}^2
+\|\Lambda^{s-1} \te \|_{L^2}^2 
\,\right)^{\frac{1}{2}}(t) 
\]
and
\[
        Y_s(t) \simeq \left(\alpha \|\Lambda^s u\|_{L^2}^2
+ \lomu\|\Lambda^{s+1} u\|_{L^2}^2 
+ \frac{\beta}{2} ||\Lambda^s v\|_{L^2}^2 
+ \alpha \|\Lambda^{s-1} u\|_{L^2}^2 
+ \lomu\|\Lambda^{s} u\|_{L^2}^2
+ \frac{\beta}{2} \|\Lambda^{s-1} v \|_{L^2}^2
+ \frac{\kappa}{2} \|\Lambda^s \te\|_{L^2}^2
\right)^{\frac{1}{2}}(t). 
\]
\end{remark}
\begin{proof}
We shall use the smallness assumed in \eqref{E54}, so that we show that 
\[
    \sum_{k=1}^{12}\cI_k \leq \epsilon C(Y_s(t))^2.
\]
In particular, \eqref{E54} yields that
\[
    \|\nabla u\|_{L^2},\|\nabla v\|_{L^2},\|\nabla \theta \|_{L^2}, \|v\|_{L^{\infty}} \lesssim \epsilon,
\]
and we shall frequently use these estimates.
To see this, we use \eqref{E14}, Lemma~\ref{L23}, and $\dv u=0$ to obtain
\begin{align*}
\cI_1
&=
\left| \int (\Lambda^s (u_i \partial_i u_j) - u_i \Lambda^s \partial_i u_j )\Lambda^s u_j\,\right|\\
&
\lesssim 
\norm{\nabla u}_{L^2}
\norm{\Lambda^s u}_{L^4}^2 
\lesssim \norm{\nabla u}_{L^2}
\norm{\Lambda^{s + \frac{1}{2}} u}_{L^2}^2 
\lesssim
\norm{\nabla u}_{L^2}
\norm{\Lambda^{s + 1} u}_{L^2}
\norm{\Lambda^s u}_{L^2}
\lesssim
\epsilon 
\norm{\Lambda^{s + 1} u}_{L^2}
\norm{\Lambda^s u}_{L^2}, \\
\cI_2
&= \lt|\int \Lambda^s (v_j v_i) \p_i \Lambda^s u_j\, \rt|
\lesssim \norm{\Lambda^s v}_{L^2}
\norm{v}_{L^\infty}
\norm{\Lambda^{s+1} u}_{L^2} \lesssim \epsilon \norm{\Lambda^s v}_{L^2}
\norm{\Lambda^{s+1} u}_{L^2},\\
\cI_3
&= \lt| \int (\Lambda^s (u_i \p_i v_j) - u_i \Lambda^s \p_i v_j) \Lambda^s v_j \, \rt| 
\lesssim \norm{\nabla u}_{L^\infty}
\norm{\Lambda^s v}_{L^2}^2
+
\norm{\Lambda^s u}_{L^\frac{2}{s-1}}
\norm{\nabla v}_{L^\frac{2}{2-s}}
\norm{\Lambda^s v}_{L^2}
\\
&\lesssim (\norm{\nabla u}_{L^2}+\norm{\Lam^{s+1} u}_{L^2})
\norm{\Lambda^s v}_{L^2}^2
\lesssim 
\epsilon
(\norm{\Lambda^s v}_{L^2}^2 
+
\norm{\Lambda^s v}_{L^2}
\norm{\Lam^{s+1} u}_{L^2}),\\
\cI_4
&= \lt|\int_{\bR^2} \Lambda^s (v_i \partial_i u_j) \Lambda^s v_j\,\rt|
\lesssim (\norm{\Lambda^s v}_{L^2} \norm{\nabla u}_{L^\infty}
+ \norm{v}_{L^\infty}
\norm{\Lambda^{s+1} u}_{L^2})
\norm{\Lambda^s v}_{L^2} \\
&\lesssim(\norm{\nabla u}_{L^2}+\norm{\Lam^{s+1} u}_{L^2})
\norm{\Lambda^s v}_{L^2}^2 +
\norm{v}_{L^\infty}
\norm{\Lambda^{s+1} u}_{L^2}
\norm{\Lambda^s v}_{L^2}
\lesssim \epsilon 
\lt(\norm{\Lambda^s v}_{L^2}^2 + 
\norm{\Lambda^{s+1}u}_{L^2} 
\norm{\Lambda^s v}_{L^2}\rt).
\end{align*}

Moreover, for $\cI_5$, we divide the cases depending on $\alpha$.
\begin{align*}
\cI_5
&= \lt| \int ( \Lambda^{s-1} (u_i \p_i u_j)
- u_i \p_i \Lambda^{s-1} u_j)
\Lambda^{s-1} u_j \,\rt| 
\lesssim 
\norm{\nabla u}_{L^2} \norm{\Lambda^{s-1} u}_{L^4}^2 
\lesssim \norm{\nabla u}_{L^2} \norm{\Lambda^{s - \frac{1}{2}} u}_{L^2}^2 \\
&
\lesssim 
\begin{cases}
\norm{\nabla u}_{L^2} 
\norm{\Lambda^{s} u}_{L^2}
\norm{\Lambda^{s-1} u}_{L^2}
\lesssim 
\epsilon 
\norm{\Lambda^{s} u}_{L^2}
\norm{\Lambda^{s-1} u}_{L^2}
\qquad \mbox{ if } \alpha>0, \\
\norm{\Lambda^s u}_{L^2}^\frac{1}{s}
\norm{u}_{L^2}^{1 - \frac{1}{s}}
\norm{\Lambda^s u}_{L^2}^{2 - \frac{1}{s}}
\norm{u}_{L^2}^\frac{1}{s}
\lesssim \norm{\Lambda^s u}_{L^2}^2
\norm{u}_{L^2}\lesssim \epsilon \norm{\Lambda^s u}_{L^2}^2
\quad \mbox{ if } \alpha=0.
\end{cases}
\end{align*}

One can see that the terms $\cI_6-\cI_9$ can be controlled similarly as follows:
\begin{align*}
\cI_6
&= \lt| \int \Lambda^{s-1} (v_j v_i) \p_i \Lambda^{s-1} u_j\,\rt|
\lesssim \norm{\Lambda^{s-1} v}_{L^2}
\norm{v}_{L^\infty}
\norm{\Lambda^{s} u}_{L^2}
\lesssim \epsilon
\norm{\Lambda^{s-1} v}_{L^2}\norm{\Lambda^s u}_{L^2},\\
\cI_{7}
&= \lt|\int ( \Lambda^{s-1} (u_i \p_i v_j)
- u_i \Lambda^{s-1} \p_i v_j)
\Lambda^{s-1} v_j \rt| 
\, \lesssim \,
\norm{\nabla u}_{L^2}
\norm{\Lambda^{s-1} v}_{L^4}^2
+
\norm{\Lambda^{s-1} u}_{L^\frac{2}{s-1}}
\norm{\nabla v}_{L^{\frac{2}{2-s}}}
\norm{\Lambda^{s-1} v}_{L^2}
\\
&\lesssim
\norm{\nabla u}_{L^2}
\norm{\Lambda^s v}_{L^2}
\norm{\Lambda^{s - 1} v}_{L^2} \lesssim \ep 
\norm{\Lambda^s v}_{L^2}
\norm{\Lambda^{s - 1} v}_{L^2},\\
\cI_{8}
&=\lt|\int_{\bR^2} \Lambda^{s-1} (v_i \partial_i u_j) \Lambda^{s-1} v_j\,\rt|
\lesssim \norm{v}_{L^\infty}
\norm{\Lam^s u}_{L^2}
\norm{\Lam^{s-1} v}_{L^2}
+ \norm{\Lam^{s-1} v}_{L^\frac{2}{s-1}}
\norm{\nabla u}_{L^\frac{2}{2-s}}
\norm{\Lam^{s-1} v}_{L^2}
\\
&
\lesssim
(\norm{v}_{L^\infty}
+ \norm{\nabla v}_{L^2})
\norm{\Lam^s u}_{L^2}
\norm{\Lam^{s-1} v}_{L^2}
\lesssim \ep
\norm{\Lam^s u}_{L^2}
\norm{\Lam^{s-1} v}_{L^2}
,
\\
\cI_9
&= \lt| \int (\Lambda^s (u_i \p_i \theta) - u_i \Lambda^s \p_i \theta) \Lambda^s \theta\,\rt| \lesssim \norm{\nabla u}_{L^\infty}
\norm{\Lambda^s \theta}_{L^2}^2
+
\norm{\Lambda^s u}_{L^\frac{2}{s-1}}
\norm{\nabla \theta}_{L^\frac{2}{2-s}}
\norm{\Lambda^s \theta}_{L^2}
\\
&
\lesssim \lt(\norm{\Lambda^{s+1} u}_{L^2}+ 
\norm{\nabla u}_{L^2}\rt)
\norm{\Lambda^s \theta}_{L^2}^2
\lesssim \ep\lt(
\norm{\Lambda^{s+1} u}_{L^2}
\norm{\Lambda^s \theta}_{L^2}
+
\norm{\Lambda^s \theta}_{L^2}^2
\rt).
\end{align*}

For $\cI_{10}$ and $\cI_{12}$, we use \eqref{E70} and \eqref{E71}. Additionally, for $\cI_{11}$, we divide the cases depending on $\alpha$. It follows that
\begin{align*}
\cI_{10}
&=\lt|\int_{\bR^2}
\left(
\Lam^s (\mu(\te)\partial_i u_j) - 
\mu(\te)\Lam^s \partial_i u_j \right)
\Lam^s \partial_i u_j\, \rt|\\
&=\lt|\int_{\bR^2}
\left(
\Lam^s ((\mu(\te)-\mu(0))\partial_i u_j) - 
(\mu(\te)-\mu(0))\Lam^s \partial_i u_j \right)
\Lam^s \partial_i u_j\, \rt| \\
&\lesssim
\norm{\Lambda^s (\mu(\te)-\mu(0))}_{L^2}
\norm{\nabla u}_{L^\infty}
\norm{\Lambda^{s+1} u}_{L^2}
+
\norm{\nabla \mu (\theta)}_{L^\frac{2}{2-s}}
\norm{\Lambda^{s} u}_{L^\frac{2}{s-1}}
\norm{\Lambda^{s+1} u}_{L^2}
\\
&\lesssim
(1+\|\nabla \te\|_{L^2}^{\lceil s-1 \rceil})
\norm{\Lambda^s \theta}_{L^2}
\lt(\norm{\Lambda^{s+1} u}_{L^2}+
\norm{\nabla u}_{L^2}\rt)
\norm{\Lambda^{s+1} u}_{L^2}
+
\norm{\nabla \theta}_{L^\frac{2}{2-s}}
\norm{\Lambda^{s} u}_{L^\frac{2}{s-1}}
\norm{\Lambda^{s+1} u}_{L^2}
\\
&\lesssim
(1 + \epsilon^s)
\norm{\Lam^s \theta}_{L^2}
\lt( \norm{\Lambda^{s+1} u}_{L^2}
+ \norm{\nabla u}_{L^2}\rt)
\norm{\Lambda^{s+1} u}_{L^2}
\lesssim
\epsilon
\lt( \norm{\Lambda^{s+1} u}_{L^2}^2
+ 
\norm{\Lam^s \theta}_{L^2}
\norm{\Lambda^{s+1} u}_{L^2}\rt),\\
\cI_{11}
&= \lt| \int (\Lambda^{s-1} (u_i \p_i \theta) - u_i \Lambda^{s-1} \p_i \theta) \Lambda^{s-1} \theta\,\rt| \\
&\lesssim
\norm{\nabla u}_{L^\frac{2}{2-s}}
\norm{\Lambda^{s-1} \theta}_{L^\frac{2}{s-1}}
\norm{\Lambda^{s-1} \theta}_{L^2}
+
\norm{\Lambda^{s-1} u}_{L^2}
\norm{\nabla \theta}_{L^\frac{2}{2-s}}
\norm{\Lambda^{s-1} \theta}_{L^\frac{2}{s-1}}
\\
&
\lesssim
\norm{\Lambda^s u}_{L^2}
\norm{\nabla \theta}_{L^2}
\norm{\Lambda^{s-1} \theta}_{L^2}
+
\norm{\Lambda^{s-1} u}_{L^2} 
\norm{\Lambda^s \theta}_{L^2}
\norm{\nabla \theta}_{L^2}
\\
&\lesssim 
\norm{\Lambda^{s} u}_{L^2} 
\norm{\Lambda^s \theta}_{L^2}
\norm{\theta}_{L^2}
+
\begin{cases}
\norm{\Lambda^{s-1} u}_{L^2} 
\norm{\Lambda^{s} \te}_{L^2}
\norm{\nabla \te}_{L^2}
\qquad \mbox{ if } \alpha>0, \\
 \norm{\Lambda^s u}_{L^2}^{1-\frac{1}{s}}
\norm{u}_{L^2}^{\frac{1}{s}}\norm{\Lambda^s \te}_{L^2}^{1+\frac{1}{s}}\norm{\te}_{L^2}^{1-\frac{1}{s}} \quad \mbox{ if } \alpha=0,
\end{cases}
\\
&\lesssim \ep
\norm{\Lambda^{s} u}_{L^2} 
\norm{\Lambda^s \theta}_{L^2}
+
\begin{cases}
\ep\norm{\Lambda^{s-1} u}_{L^2} 
\norm{\Lambda^{s} \te}_{L^2}
\qquad \mbox{ if } \alpha>0, \\
\ep \norm{\Lambda^s u}_{L^2}^{1-\frac{1}{s}}\norm{\Lambda^s \te}_{L^2}^{1+\frac{1}{s}}
\quad \mbox{ if } \alpha=0,
\end{cases}
\\
\cI_{12}
&=\lt|\int_{\bR^2}
\left(
\Lam^{s-1} (\mu(\te)\partial_i u_j) - 
\mu(\te)\Lam^{s-1} \partial_i u_j \right)
\Lam^s \partial_i u_j\, \rt|\\
&\lesssim
\lt(\norm{\Lambda^{s-1} \mu (\theta)}_{L^\frac{2}{s-1}}
\norm{\nabla u}_{L^\frac{2}{2-s}}
+ \norm{\nabla \theta}_{L^\frac{2}{2-s}}
\norm{\Lambda^{s-1} u}_{L^\frac{2}{s-1}}\rt)
\norm{\Lambda^{s} u}_{L^2}
\\
&\lesssim
\lt(\norm{\nabla \mu (\theta)}_{L^2}
\norm{\Lam^s u}_{L^2}
+ \norm{\Lambda^s \theta}_{L^2}
\norm{\nabla u}_{L^2}\rt)
\norm{\Lambda^{s} u}_{L^2}
\\
&\lesssim
\lt(\norm{\nabla \theta}_{L^2}
\norm{\Lam^s u}_{L^2}
+ \norm{\Lambda^s \theta}_{L^2}
\norm{\nabla u}_{L^2}\rt)
\norm{\Lam^s u}_{L^2}
\lesssim
\ep \lt(\|\Lambda^{s}u\|_{L^2}^2 + \|\Lambda^{s}\te\|_{L^2}\|\Lambda^{s}u\|_{L^2}\rt).
\end{align*}

Lastly, using \eqref{Power1}, the terms $\cI_{13}$, $\cI_{14}$, and $\cI_{15}$ can be expressed as 
\begin{equation}
\label{E47a}
\begin{split}
\cI_{13}+\cI_{14}+\cI_{15}
&=\int_{\bR^2} \left( \Lam^{s-1} (u_i \p_i v_j) 
- u_i \p_i \Lam^{s-1} v_j \right) \p_j \Lam^{s-1} \theta
-  \int_{\bR^2} \left( \Lam^{s-1} (u_i \p_i \theta)
- u_i \p_i \Lam^{s-1} \theta \right) \Lam^{s-1} \p_j v_j \\
&\quad + \int_{\bR^2}  (\p_i \Lam^{s-1} \theta) (\p_j u_i)(\Lam^{s-1} v_j)
+ \int_{\bR^2} \Lam^{s-1} (v \cdot \nabla)u \cdot \Lam^{s-1}\nabla \theta,
\end{split}
\end{equation}
and the similar calculation shows that these terms can be bounded by
\begin{equation*}
    \epsilon C(\|\Lambda^s u\|_{L^2}\|\Lambda^s v\|_{L^2} + \|\Lambda^s v\|_{L^2}\|\Lambda^s \te\|_{L^2} + \|\Lambda^s \te\|_{L^2}\|\Lambda^s u\|_{L^2}).
\end{equation*} 
Hence, we 
combine all of these and \eqref{H^s_temp} with \eqref{X_s(t)}-\eqref{Y_s(t)} to conclude that
\[
\frac{1}{2}\frac{\ud}{\dt}\lt(X_s(t)\rt)^2 
+ Y_s(t)^2 \leq \epsilon C Y_s(t)^2.
\]
We choose $\epsilon_3>0$ sufficiently small such that $0 < \epsilon < \epsilon_3$ implies $\epsilon C \leq \frac{1}{2}$. 
Hence, we obtain 
\begin{equation*}
\frac{\ud}{\dt} X_s(t)^2 
+ Y_s(t)^2 \le 0,
\end{equation*}
which completes the proof.
\end{proof}

\begin{remark}
In the process of choosing $\epsilon_3$ to satisfy $\epsilon C \leq \frac{1}{2}$, we emphasize that the $C$ depends only on $s,\al, \be, \lomu$, and $\|\mu'\|_{C^{s-1}}$. Hence, the small constants $\epsilon_3$ stated in Lemma~\ref{L33} can be fixed accordingly, with their dependence inherited from that of $C$.
\end{remark}

\subsection{Proof of key proposition}\label{subsec_33}
Now, we combine Lemma~\ref{L31} and Lemma~\ref{L33} to conclude Proposition~\ref{key_prop}.

\begin{proof}[Proof of Proposition \ref{key_prop}]
Since the main system \eqref{E11} is locally well-posed in $H^s(\bR^2)$, given $(u_0, v_0, \te_0) \in H^s(\bR^2)$, there exists $T>0$ such that solution $(u,v,\te)$ uniquely exists in $\cC([0,T];H^s(\bR^2))$. We denote $T_{\max}$ by the maximal time of existence of \eqref{E11}.
By integrating \eqref{ingredi_direct-2} over $[0,t]$, we finally deduce from Lemma~\ref{L31} that
\begin{equation}
\label{result1}
    \sup_{\tau \in [0,t]}\lt(A_s(\tau)\rt)^2 + \int_0^t \lt(B_s(\tau)\rt)^2 \ud \tau \leq \lt(A_s(0)\rt)^2+ C \sup_{\tau \in [0,t]}\left(A_s(\tau)
+ \lt(A_s(\tau)\rt)^{s+1}\right)\cdot \int_0^t\lt(B_s(\tau)\rt)^2\ud \tau,
\end{equation}
where $C$ is a constant which depends only on $\al,\beta,\lomu$,$s$, and $\norm{ \mu'}_{C^{s-1}}$. To compute precisely, we set
\[
    \widetilde{A}_s(t):=\left( 
\|\Lambda^s u(t)\|_{L^2}^2
+\|\Lambda^{\delta_1} u(t)\|_{L^2}^2
+\| v(t)\|_{H^s}^2
+\|\te(t)\|_{H^s}^2
\right)^{\frac{1}{2}}.
\]
Then, it is easily shown that (c.f. Remark~\ref{rmk315})
\begin{equation}
\label{E77}
       \frac{3}{4}A_s(t)^2 \leq \widetilde{A}_s(t)^2 \leq \frac{5}{4}A_s(t)^2.
    \end{equation}
Below we set
\[
    E_s(t) := \sup_{\tau \in [0,t]}\lt(\widetilde{A}_s(\tau)\rt)^2 + \int_0^t (B_s(\tau))^2 \ud \tau.
\]
Then, \eqref{result1} and \eqref{E77} together yield that
\begin{equation}
\label{E_eqn}
    E_s(t) \leq \frac{5}{3}E_s(0) 
    + C\,E_s(t)^{\frac{3}{2}} 
    + C\,E_s(t)^{\frac{s+3}{2}},
\quad \text{for } t \in [0,T_{\max}).
\end{equation}
By  standard argument, there exists a small number $\epsilon_4$ such that if $0<\epsilon < \ep_4$ and $E_s(0) <\epsilon$, then $E_s(t) \leq 2\epsilon$ for all $t\geq0$.
To prove the details, we choose $\ep_4:= \min\left(\frac{1}{3(22C)^2}, \frac{1}{3(22C)^{\frac{2}{s+1}}}\right)$. Suppose that $0<\epsilon<\epsilon_4$ and $E_s(0)<\epsilon$. It is sufficient to show that
\begin{equation}
\label{E324}
    E_s(t) < 2\epsilon \quad \text{for all } t<T_{\max}.
\end{equation}Arguing by contradiction, suppose that \eqref{E324} is false. Then, by the time-continuity of the solution, there exists $0<t_1<T_{\max}$ such that
\begin{equation}
\label{E325}
    2 \epsilon \leq E_s(t_1) <3\epsilon.
\end{equation} 
Since we have chosen $\ep \leq \min\left(\frac{1}{3(22C)^2}, \frac{1}{3(22C)^{\frac{2}{s+1}}}\right)$, it follows that
\begin{equation}
\label{ep_condition}
    CE_s(t_1)^{\frac{3}{2}}+CE_s(t_1)^{\frac{s+3}{2}} \leq \frac{1}{11}E_s(t_1).
\end{equation}
However, \eqref{E325}-\eqref{ep_condition} contradicts to \eqref{E_eqn}.
Indeed, the following estimate follows from \eqref{E_eqn} with $E_s (0) \le \ep$ that for some $0<t_1<T_{\max}$
\[
E_s (t_1)
\le \frac{11}{6} \ep,
\]
which contradicts to $2 \ep \le E_s (t_1)$ in \eqref{E325}.
Therefore, \eqref{E324} is established, which implies that $T_{\max}=\infty$. Hence, we have established \eqref{small:no_damping} and \eqref{small:damped}, for both cases $\alpha=0$ or $\al>0$.

Now, we define
\[
\varepsilon^*:=\min(\ep_3,\ep_4)>0,
\]
where $\ep_3$ is the number stated in Lemma~\ref{L33}.
Then, it only remains to prove the temporal decay estimates, i.e. \eqref{s-1:decay}.
Again, in the process of approximating by mollification, we combine the assumptions \eqref{Assump1} and \eqref{Assump2} with the choice of $\varepsilon^* \leq \epsilon_4 $ in the above procedure to deduce that the smooth solution $(u,v,\te)$ satisfies the assumption of Lemma~\ref{L33}.
This leads to \eqref{E33}. 

Note that, since $\kappa>0$ is chosen to be less than $\frac{1}{2}$, it follows that
\[
    \frac{1}{2}X_s(t) \leq \left( \|\Lambda^s u\|_{L^2}^2
+ \|\Lambda^s v\|_{L^2}^2
+ \|\Lambda^s \te\|_{L^2}^2 
+ \|\Lambda^{s-1} u\|_{L^2}^2
+ \|\Lambda^{s-1} v\|_{L^2}^2
+\|\Lambda^{s-1} \te \|_{L^2}^2
\,\right)
\leq 2 X_s(t).
\]
Moreover, there exists a constant $K>0$ such that
\begin{equation}
\label{E34}
X_s(t)^\frac{2s}{s-1}
\le \frac{1}{K} Y_s(t)^2.
\end{equation}
For example, by using \eqref{E21}, we can show that 
\[
\norm{\Lam^{s-1} \theta}_{L^2}^2
\le \norm{\Lam^s \theta}_{L^2}^{2-\frac{2}{s}}
\norm{\theta}_{L^2}^\frac{2}{s}.
\]
Then, this implies
\[
\norm{\Lam^{s-1} \theta}_{L^2}^\frac{2s}{s-1}
\|(u_0,v_0,\te_0)\|_{L^2}^\frac{-2}{s-1}
\le (\norm{\Lam^{s-1} \theta}_{L^2}^2
\norm{\theta}_{L^2}^\frac{- 2}{s})^\frac{s}{s-1}
\le \norm{\Lam^s \theta}_{L^2}^2.
\]
Similarly, we can use the bound
\[
    \|\Lambda^su\|_{L^2}^2 = 
    \|\Lambda^su\|_{L^2}^{2-\frac{2}{s}}
    \|\Lambda^su\|_{L^2}^{\frac{2}{s}} \leq
    \|\Lambda^su\|_{L^2}^{2-\frac{2}{s}}(2\ep)^{\frac{2}{s}}.
\]
Proceeding in this way, we can bound the remaining terms using simple calculations and thus deduce \eqref{E34}. Therefore, we derive the following estimate
\[
\frac{\ud}{\dt} X_s(t)^2
+ K X_s(t)^\frac{2s}{s-1} \le 0.
\]
Integrating in time over $(0,t)$ and after a straightforward computation, we obtain
\[
X_s(t)
\le \left( \frac{X_s(0)^\frac{2}{s-1}}{t X_s(0)^\frac{2}{s-1} \frac{K}{s-1} + 1} \right)^{\frac{s-1}{2}}.
\]
Therefore, we arrive at
\[
\norm{\Lambda^{s-1} u(t)}_{L^2}
+ \norm{\Lambda^{s-1} v(t)}_{L^2}
+ \norm{\Lambda^{s-1} \theta(t)}_{L^2}
\le 2X_s(t)
\le C (1 + t)^{- \frac{s-1}{2}},
\]
where the constant $C$ is independent of $t$, and depends only on $s,\al,\be,\lomu, \|\mu'\|_{C^{s-1}}$, and $\|(u_0,v_0,\te_0)\|_{H^s}$. 
Finally, the a priori estimate concerning the temporal decay estimates has been established.
This allows us to pass to the limit in the approximation scheme, thereby completing the proof.
\end{proof}
\begin{remark}
    We remark that one can also obtain
\begin{equation}
\label{E57}
    \norm{\Lambda^{s} u(t)}_{L^2}
+ \norm{\Lambda^{s} v(t)}_{L^2}
+ \norm{\Lambda^{s} \theta(t)}_{L^2}
\le 2X_s(t)
\le C (1 + t)^{- \frac{s-1}{2}}.
    \end{equation}
This will be used later.
\end{remark}

\section{Estimates on the spaces of higher regularity}
\label{S4}
\subsection{Stability estimate}  In this section, we prove a proposition concerning the stability in higher order Sobolev spaces.  
Without loss of generality, it suffices to assume that $1<s<2$. Again, we employ the number $\lambda$ in \eqref{lambda}. 
\begin{proposition} 
\label{P33}
Let $\alpha \geq 0$, and $s \in(1,2)$, and suppose that the initial data $(u_0,v_0,\theta_0)$ belong to $\cC^{\infty}_0(\bR^2)$.  For a number $\varepsilon^*>0$ stated in Proposition~\ref{key_prop}, there holds followings: 
\begin{enumerate}
    \item (Undamped case) In the case of $\alpha=0$, suppose that $0<\epsilon<\varepsilon*$ and
\[
\|u_0\|_{H^s(\bR^2)}+\|v_0\|_{H^s(\bR^2)}+\|\te_0\|_{H^s(\bR^2)} < \ep.
\]
Then, the global solution $(u,v,\te)$ of the Cauchy problem \eqref{E11} constructed in Proposition~\ref{key_prop} belong to $\mathcal{C}^{\infty} ([0,\infty); \mathcal{C}^{\infty} (\bR^2))$. For each $m \geq s$, there exists a constant $C$ depends on $m, \epsilon, s, \beta$, and $\|\mu'\|_{C^{m-1}(\bR)}$  such that
\begin{align*}
\sup_{t \in [0,\infty)}\left(\|u(t)\|_{H^m(\bR^2)}
+ \|v(t)\|_{H^m(\bR^2)}
+ \|\te(t)\|_{H^m(\bR^2)} \right)
+&\lambda \left( \int_0^{\infty}  \|\nabla u(t)\|_{H^m}^2 +\|v(t)\|_{H^m}^2+\|\nabla \te(t)\|_{H^{m-1}}^2  \,\ud t \right)^{\frac{1}{2}} \\ 
&\qquad\leq C( \|u_0\|_{H^m(\bR^2)}+\|v_0\|_{H^m(\bR^2)}+\|\te_0\|_{H^m(\bR^2)}).
\end{align*}
    \item (Damped case) In case of $\alpha >0$, suppose that $0<\epsilon<\varepsilon*$ and
\[
\|u_0\|_{\dot{H}^s\cap\dot{H}^1}+\|v_0\|_{H^s}+\|\te_0\|_{H^s} < \epsilon.
\]
Then, the global solution $(u,v,\te)$ of the Cauchy problem \eqref{E11} constructed in Proposition~\ref{key_prop} belong to $\mathcal{C}^{\infty} ([0,\infty); \mathcal{C}^{\infty} (\bR^2))$.  Indeed, for each $m \geq s$, there exists a constant $C$ depends on $m, \epsilon, s, \al, \beta$, and $\|\mu'\|_{C^{m-1}(\bR)}$ such that
\begin{align*}
    \sup_{t \in [0,\infty)}\left( \|u(t)\|_{\dot{H}^m\cap\dot{H}^{1}}
    +\|v(t)\|_{H^m}
    +\|\te(t)\|_{H^m} \right)+&
    \lambda\left( \int_0^{\infty}  \|\nabla u(t)\|_{H^m}^2 +\|v(t)\|_{H^m}^2+\|\nabla \te(t)\|_{H^{m-1}}^2  \,\ud t \right)^{\frac{1}{2}}\\
    & \qquad \leq C( \|u_0\|_{\dot{H}^m\cap\dot{H}^1(\bR^2)}+\|v_0\|_{H^m(\bR^2)}+\|\te_0\|_{H^m(\bR^2)}).
\end{align*}
\end{enumerate}

\end{proposition}

\begin{proof} 
In view of Proposition~\ref{key_prop}, there exist a unique solution $(u,v,\te)$ in $C([0,\infty);H^s(\bR^2))$. The standard blow-up criterion reveals that $H^m$-regularity of the solution is preserved over time, for any $m \geq s$. Hence, we conclude that the solution $(u,v,\te)$ is indeed in $C^{\infty}([0,\infty);C^{\infty}(\bR^2))$.

We recall $\delta_1$ defined in \eqref{delta_1},
    \begin{equation*}
        {\delta_1} =
        \begin{cases}
        0 & \text{if } \alpha =0,\\
        1 & \text{if } \alpha >0.    
        \end{cases}
    \end{equation*}
The following identity is obtained by substituting $m$ for 
$s$ in the expression previously used in \eqref{E12}.
Here, we have chosen $\eta$ to satisfy $0 < \eta < \beta/(4+2\beta^2) < \frac{1}{4}$. 
\begin{equation}
\label{E13}
\begin{split}
&\frac{1}{2} \frac{\ud}{\dt}
\left( \|\Lambda^m u\|_{L^2}^2
+ \|\Lambda^m v\|_{L^2}^2
+ \|\Lambda^m \te\|_{L^2}^2 
+ \|\Lambda^{\delta_1} u\|_{L^2}^2
+ \| v\|_{L^2}^2
+\| \te \|_{L^2}^2
-\eta  \int_{\bR^2} \Lam^{m-1} v \cdot \Lam^{m-1}\nabla \theta
-\eta\int_{\bR^2}  v \cdot \nabla \theta
\right) \\
&\quad + \left(\alpha \|\Lambda^m u\|_{L^2}^2
+ \lomu\|\Lambda^m \nabla u\|_{L^2}^2 
+ \frac{\beta}{2} ||\Lambda^m v\|_{L^2}^2 
+ \alpha \|\Lambda^{\delta_1} u\|_{L^2}^2 
+ \lomu\|\Lambda^{\delta_1} \nabla u\|_{L^2}^2
+ \frac{\beta}{2} \| v \|_{L^2}^2\right)
+ \frac{\eta}{2} \|\nabla \te\|_{H^{m-1}}^2 \\
&\leq - \int_{\bR^2} \Lambda^m ((u \cdot \nabla) u) \cdot \Lambda^m u
+ \int_{\bR^2} \Lambda^m (v \otimes  v) : \Lambda^m \nabla u 
- \int_{\bR^2} \Lambda^m ((u \cdot \nabla) v) \cdot \Lambda^m v
- \int_{\bR^2} \Lambda^m ((v \cdot \nabla) u) \cdot \Lambda^m v\\
&\quad - \int_{\bR^2} \Lambda^{ \delta_1} ((u \cdot \nabla) u) \cdot \Lambda^{ \delta_1} u
+ \int_{\bR^2} \Lambda^{ \delta_1} (v \otimes  v) : \Lambda^{ \delta_1}  \nabla u
- \int_{\bR^2} (v \cdot \nabla ) u \cdot v  - \int_{\bR^2} \Lambda^m (u \cdot \nabla \te) \Lambda^m \te \\
&\quad -\int_{\bR^2}\left(\Lam^m (\mu(\te)\nabla u) - \mu(\te)\Lam^m \nabla u \right):\Lam^m \nabla u
-\int_{\bR^2}\left(\Lam^{\delta_1} (\mu(\te)\nabla u) - \mu(\te)\Lam^{\delta_1} \nabla u \right):\Lambda^{\delta_1}\nabla u \\
&\quad +
\eta\int_{\bR^2} \Lam^{m-1} (u \cdot \nabla)v \cdot \Lam^{m-1}\nabla \theta 
+ \eta\int_{\bR^2} \Lam^{m-1} (v \cdot \nabla)u \cdot \Lam^{m-1}\nabla \theta 
- \eta\int_{\bR^2} \Lam^{m-1} ( u \cdot \nabla \theta) \cdot \Lam^{m-1}\dv v \\
&\quad + 2\eta \int (v\cdot \nabla)u \cdot \nabla \te\\
&=: \sum_{k=1}^{14}\cJ_k.
\end{split}
\end{equation}
Thus, for each time $t>0$, we define $A_m(t)$ and $B_m(t)$ by
\begin{equation}
\label{A_m(t)}
A_m(t) := 
\left( 
\|\Lambda^m u(t)\|_{L^2}^2+
\|\Lambda^{\delta_1} u(t)\|_{L^2}^2+
\| v(t)\|^2_{H^m}+
\|\te(t)\|^2_{H^m}
-\eta  \int_{\bR^2} \Lam^{m-1} v(t) \cdot \Lam^{m-1}\nabla \theta(t)
-\eta\int_{\bR^2}  v(t) \cdot \nabla \theta(t)
\right)^{\frac{1}{2}} 
\end{equation}
and
\begin{equation}
\label{B_m(t)}
B_m(t) := \lambda\left( \|\nabla u(t)\|_{\dot{H}^m}^2 
+ \| v(t) \|^2_{H^m} + \|\Lambda^m \te(t)\|_{L^2}^2 \right)^{\frac{1}{2}}. 
\end{equation}
Recalling $B_s(t)$ defined in \eqref{B_s(t)}, 
our claim is that
there exists a constant $C>0$ which depends on $m,s,\al, \be$, and $\lomu$ such that
\begin{equation}\label{E78}
\frac{\ud}{\dt}A_m^2 + B_m^2 \leq C\lt(B_s^2A_m^2\rt). 
\end{equation}

By the assumption, we can apply Proposition ~ \ref{key_prop}.
In particular, notice that
\begin{equation}
\label{stab_timeL2}
\frac{1}{\lambda}
\int_0^{\infty}
B_s (t)^2 \dt 
=
\int_0^{\infty}
\lt(\|\nabla u(t)\|_{H^s(\bR^2)}^2
+\|v(t)\|_{H^s(\bR^2)}^2
+\|\Lambda^s\te(t)\|_{L^2(\bR^2)}^2\rt)\dt \lesssim \epsilon.
\end{equation}
Thus, if the claim \eqref{E78} is shown, then we use Gr\"onwall's inequality to obtain
\begin{equation*}
    A_m(t) \leq A_m(0)e^{C\int_0^t\lt(B_s(\tau)\rt)^2\ud\tau} \leq  A_m(0)e^{C\epsilon}.
\end{equation*}
Moreover, by integrating \eqref{E78} over $[0,t]$, we have
\begin{align*}
\int_0^t \lt(B_m(\tau)\rt)^2 \ud \tau  
\leq \lt(A_m(t)\rt)^2 + 
\int_0^t \lt(B_m(\tau)\rt)^2 \ud \tau 
&\leq \lt(A_m(0)\rt)^2 + 
\sup_{[0,t]}\lt(A_m(\tau)\rt)^2\int_0^t \lt(B_s(\tau)\rt)^2 \ud \tau\\
& \leq (1+\tilde{C}e^{C\epsilon}\epsilon)A_m(0)^2.
\end{align*}
Combining these, we deduce that
\begin{equation*}
    \sup_{t \in [0,\infty)}A_m(t) + \lt(\int_0^{\infty}\lt(B_m(t)\rt)^2\,\dt\rt)^{\frac{1}{2}} \leq C^*A_m(0),
\end{equation*}
where $C^* = 1+\cO(\epsilon)$. 
Then, the proof is finished.

\vspace{1em}

Hence, our goal is to establish the estimate $\cJ_k$ in \eqref{E13} in the form of
\[
\sum_{k=1}^{13}\cJ_k \leq C \lt(B_s(t)A_m(t)B_m(t) + \lt(B_s(t)A_m(t)\rt)^2\rt) + \frac{1}{4} B_m(t)^2.
\]
If the above is shown, Young's inequality and the absorption scheme conclude \eqref{E78}.

For the terms $\cI_5, \cI_6,\cI_7$ and $\cI_{10}$, we consider two cases, either $ \alpha = 0$ or $ \alpha > 0$. 

Considering the case where $ \alpha = 0$ (i.e., $ \delta_1 = 0$), we proceed by integration by parts to obtain:
\begin{equation*}
\cJ_5 = \int_{\bR^2} (u \cdot \nabla u) \cdot u = 0
\quad \text{and} \quad 
\cJ_6 + \cJ_7 = \int_{\bR^2} (v \otimes  v) : \nabla u
- \int_{\bR^2} (v \cdot \nabla ) u \cdot v
= 0.
\end{equation*}
Moreover, by definition, it is obvious that $\cJ_{10}=0$.

If $\alpha >0$, equivalently $\delta_1 =1$, then we have 
\begin{align*}
\cJ_5
&\lesssim \norm{ \nabla u}_{L^{\infty}}
\norm{ \nabla u}_{L^2}
\norm{ \nabla u}_{L^2}
\lesssim \norm{ \nabla u}_{H^s}
\norm{ \nabla u}_{L^2}
\norm{ \nabla u}_{L^2}
\lesssim B_s A_m B_m, \\
\cJ_6
&\lesssim \norm{ \nabla v}_{L^2}
\norm{ v}_{L^\infty}
\norm{ \Lam^2 u}_{L^2}
\lesssim \norm{ \nabla v}_{L^2}
\norm{v}_{H^m} 
\norm{ \nabla u}_{H^m}
\lesssim B_s A_m B_m
, \\
\cJ_7
&\lesssim \norm{v}_{L^\infty}
\norm{ \nabla u}_{L^2}
\norm{v}_{L^2}
\lesssim \norm{v}_{H^s} 
\norm{ \nabla u}_{L^2}
\norm{v}_{L^2} 
\lesssim B_s A_m B_m,\\
\cJ_{10}
&\lesssim \norm{\nabla \te}_{L^2}
\norm{ \nabla u}_{L^{\infty}}
\norm{\Lambda^2 u}_{L^2}
\lesssim \norm{\te}_{H^m} 
\norm{ \nabla u}_{H^s}
\norm{\nabla u}_{H^m} 
\lesssim A_m B_s B_m,
\end{align*}
where we used $\nabla \mu (\theta) = \mu'(\theta) \nabla \theta$ in $\cJ_{10}$.

The remaining terms can be handled similarly, regardless of the presence of the damping term $\al u$. To demonstrate this, using \eqref{E14}, Lemma~\ref{L23}, Lemma~\ref{L24}, and $\dv u = 0$, we see that
\begin{align*}
\cJ_1 
&\lesssim
\norm{ \nabla u}_{L^2}
\norm{\Lambda^m u}_{L^4}
\norm{ \Lambda^m u}_{L^4} 
\lesssim
\norm{ \nabla u}_{L^2}
\norm{\Lam^{m+1} u}_{L^2}
\norm{\Lambda^m u}_{L^2}
\lesssim B_s B_m A_m
, \\
\cJ_2
&\lesssim \norm{ \Lam^m v}_{L^2}
\norm{v}_{L^\infty}
\norm{\Lam^{m+1} u}_{L^2}
\lesssim
\norm{ \Lam^m v}_{L^2}
\norm{v}_{H^s}
\norm{\Lam^{m+1} u}_{L^2}
\lesssim A_m B_s B_m
, \\
\cJ_3
&\lesssim
\norm{ \nabla u}_{L^\infty}
\norm{ \Lam^m v}_{L^2}^2
+
\norm{ \Lam^m u}_{L^\frac{2}{s-1}}
\norm{ \nabla v}_{L^\frac{2}{2-s}}
\norm{ \Lam^m v}_{L^2} 
\lesssim
\|\nabla u\|_{H^s}\|\Lambda^m v\|_{L^2}^2 +
\norm{ \Lam^{m-s+2} u}_{L^2}
\norm{ \Lam^s v}_{L^2}\norm{ \Lam^m v}_{L^2} \\
&\lesssim
\|\nabla u\|_{H^s}\|\Lambda^m v\|_{L^2}^2 +
\norm{ \nabla u}_{H^m}
\norm{ \Lam^s v}_{L^2}
\norm{ \Lam^m v}_{L^2}
\lesssim B_s B_m A_m + B_m B_s A_m,
\\
\cJ_4
&\lesssim \norm{ \Lam^m v}_{L^2}
\norm{ \nabla u}_{L^\infty}
\norm{ \Lam^m v}_{L^2}
+ \norm{v}_{L^\infty}
\norm{\Lam^{m+1} u}_{L^2}
\norm{ \Lam^m v}_{L^2} \\
&\lesssim  \norm{ \Lam^m v}_{L^2} \|\nabla u\|_{H^s}
\norm{ \Lam^m v}_{L^2}
+ 
\norm{ v}_{H^s}
\norm{\Lam^{m+1} u}_{L^2}
\norm{ \Lam^m v}_{L^2}
\lesssim B_m B_s A_m + B_s B_m A_m,
\end{align*}
and
\begin{align*}
\cJ_8
=\lt|\int_{\bR^2} \lt(\Lambda^m(u\cdot \nabla \te)-(u\cdot \nabla) \Lambda^m \te \rt): \Lambda^m \te \, \rt|
&\lesssim
\|\nabla u \|_{L^{\infty}}
\|\Lambda^m \te\|_{L^2}^2
+  
\|\Lambda^m u \|_{L^{\frac{2}{s-1}}}
\|\nabla \te\|_{L^{\frac{2}{2-s}}} 
\|\Lambda^{m} \te\|_{L^2} \\
&\lesssim 
\|\nabla u \|_{H^s}
\|\Lambda^m \te\|_{L^2}^2
+  
\|\Lambda^{m-s+2} u \|_{L^2}
\|\Lambda^s \te\|_{L^2} 
\|\Lambda^{m} \te\|_{L^2} \\
&\lesssim 
\|\nabla u \|_{H^s}
\|\Lambda^m \te\|_{L^2}^2
+  
\|\nabla u \|_{H^m}
\|\Lambda^s \te\|_{L^2} 
\|\Lambda^{m} \te\|_{L^2}\\
&
\lesssim B_s A_m B_m + B_m B_s A_m.
\end{align*}

We apply Young's inequality together with Lemma~\ref{L24}, \eqref{E70}, and \eqref{E71} to obtain 
\begin{align*}
\cJ_9
&=
\bigg|\int_{\bR^2}\left(\Lam^m (\mu(\te)\nabla u) - \mu(\te)\Lam^m \nabla u \right):\Lam^m \nabla u\bigg|\\
&=\bigg|\int_{\bR^2}\left(\Lam^m ((\mu(\te)-\mu(0))\nabla u) - (\mu(\te)-\mu(0))\Lam^m \nabla u \right):\Lam^m \nabla u\bigg|\\
&\lesssim 
\|\Lambda^m (\mu(\te)-\mu(0))\|_{L^2}
\|\nabla u \|_{L^{\infty}} 
\|\Lambda^{m+1}u\|_{L^2} 
+  
\|\nabla \mu(\te)\|_{L^{\frac{2}{2-s}}} 
\|\Lambda^m u \|_{L^{\frac{2}{s-1}}} 
\|\Lambda^{m+1}u\|_{L^2}\\
&\lesssim 
(1+\|\nabla \te\|_{L^2}^{\lceil m-1 \rceil})
\|\Lambda^m \te\|_{L^2} 
\|\nabla u \|_{H^s} 
\|\Lambda^{m+1}u\|_{L^2} 
+  
\|\nabla \te\|_{L^{\frac{2}{2-s}}} 
\|\Lambda^{m-s+2} u \|_{L^2}
\|\Lambda^{m+1}u\|_{L^2}\\
&\lesssim
(1+\|\te\|^{\lceil m-1 \rceil}_{H^s})
\|\Lambda^m \te\|_{L^2} 
\|\nabla u \|_{H^s} 
\|\Lambda^{m+1}u\|_{L^2} 
+  
\|\Lambda^s \te\|_{L^2} 
\|\Lambda^{m-s+2} u \|_{L^2}
\|\Lambda^{m+1}u\|_{L^2}.
\end{align*}
Now, we use Young's inequality with $\|\Lambda^{m-s+2} u \|_{L^2} \lesssim \|\Lambda^{m} u \|_{L^2}^{s-1} 
\|\Lambda^{m+1}u\|_{L^2}^{2-s}$ to see that
\begin{align*}
\|\Lambda^s \te\|_{L^2} 
\|\Lambda^{m-s+2} u \|_{L^2}
\|\Lambda^{m+1}u\|_{L^2}
&\lesssim 
(\|\Lambda^s \te\|_{L^2} 
\|\Lambda^{m} u \|_{L^2})^{s-1}
\|\Lambda^s \te\|_{L^2}^{2-s}
\|\Lambda^{m+1}u\|_{L^2}^{3-s} \\
&\leq 
C(\|\Lambda^s \te\|_{L^2} 
\|\Lambda^{m} u \|_{L^2})^2 + 
\frac{\underline{\mu}}{4} 
\|\Lambda^s \te\|_{L^2}^{\frac{4-2s}{3-s}}
\|\Lambda^{m+1}u\|_{L^2}^{2} \\
&\leq
C(\|\Lambda^s \te\|_{L^2} 
\|\Lambda^{m} u \|_{L^2})^2 + 
\frac{\underline{\mu}}{4} 
\ep^{\frac{4-2s}{3-s}}
\|\Lambda^{m+1}u\|_{L^2}^{2}\\
&\leq
C
(\|\Lambda^s \te\|_{L^2} 
\|\Lambda^{m} u \|_{L^2})^2 + 
\frac{\underline{\mu}}{4} 
\|\Lambda^{m+1}u\|_{L^2}^{2}.
\end{align*}
Hence, the term $\cJ_9$ can be bounded by
\[
    \cJ_9 \leq C(A_m B_s B_m) + C(B_sA_m)^2 + \frac{1}{4}B_m^2. 
\]
On the other hands, by using $\nabla \mu(\te) = \mu'(\te)\nabla \te$,
\begin{align*}
\cJ_{10}
&=\lt|\int_{\bR^2}\left(\Lam (\mu(\te)\nabla u) - \mu(\te)\Lam \nabla u \right) : \Lambda \nabla u \rt|\\
& \lesssim \|\nabla \te\|_{L^2}
\|\nabla u \|_{L^{\infty}}
\|\Lambda^2 u \|_{L^2} 
\lesssim \|\te\|_{H^m}
\|\nabla u \|_{H^s}
\|\nabla u \|_{H^m}
\lesssim A_m B_s B_m,\\
\cJ_{12}
&\lesssim \norm{ \Lam^{m-1} v}_{L^2}
\norm{ \nabla u}_{L^\infty}
\norm{ \Lam^m \theta}_{L^2}
+ \norm{v}_{L^\infty}
\norm{ \Lam^m u}_{L^2}
\norm{ \Lam^m \theta}_{L^2} \\
&\lesssim \norm{v}_{H^m}
\norm{ \nabla u}_{H^s}
\norm{ \Lam^m \theta}_{L^2}
+ \norm{v}_{H^s}
\norm{ \Lam^m u}_{L^2}
\norm{ \Lam^m \theta}_{L^2}
\lesssim A_m B_s B_m + B_s A_m B_m.
\end{align*}

To estimate $\cJ_{11}$ and $\cJ_{13}$ terms, we use cancellation properties and integration by parts, similarly to  \eqref{Power1}. It follows that
\begin{equation}\label{Power2}
    \begin{aligned}
\cJ_{11} + \cJ_{13}
&=  \int_{\bR^2} \left( \Lam^{m-1} (u_i \p_i v_j) 
- u_i \p_i \Lam^{m-1} v_j \right) \p_j \Lam^{m-1} \theta
-  \int_{\bR^2} \left( \Lam^{m-1} (u_i \p_i \theta)
- u_i \p_i \Lam^{m-1} \theta \right) \Lam^{m-1} \p_j v_j \\
&\quad + \int_{\bR^2}  (\p_i \Lam^{m-1} \theta) (\p_j u_i)(\Lam^{m-1} v_j)
=: M_1 + M_2 + M_3.
    \end{aligned}
\end{equation}
By using \eqref{E21}, \eqref{E22}, and \eqref{E24}, one has
\begin{equation*}
\begin{aligned}
M_1 
&\lesssim \norm{ \nabla u}_{L^{\infty}}
\norm{ \Lam^{m-1} v}_{L^2}
\norm{ \Lam^m \theta}_{L^2}
+
\norm{ \Lam^{m-1} u}_{L^{\frac{2}{s-1}}}
\norm{ \nabla v}_{L^{\frac{2}{2-s}}}
\norm{ \Lam^m \theta}_{L^2}  \\
& 
\lesssim \norm{ \nabla u}_{H^s}
\norm{ v}_{H^m}
\norm{ \Lam^m \theta}_{L^2}
+
\norm{ \Lambda^{m-s+1} u}_{L^2}
\norm{ \Lambda^s v}_{L^2}
\norm{ \Lam^m \theta}_{L^2} \\
&
\lesssim \norm{ \nabla u}_{H^s}
\norm{ v}_{H^m}
\norm{ \Lam^m \theta}_{L^2}
+
\norm{ \nabla u}_{H^m}
\norm{ \Lambda^s v}_{L^2}
\norm{ \Lam^m \theta}_{L^2}
\lesssim B_s A_m B_m + B_m B_s A_m,\\
M_2 &\lesssim \norm{ \nabla u}_{L^{\infty}}
\norm{ \Lam^{m-1} \te}_{L^2}
\norm{ \Lam^m v}_{L^2}
+
\norm{ \Lam^{m-1} u}_{L^{\frac{2}{s-1}}}
\norm{ \nabla \te}_{L^{\frac{2}{2-s}}}
\norm{ \Lam^m v}_{L^2} \\
& \lesssim \norm{ \nabla u}_{H^s}
\norm{ \te}_{H^m}
\norm{ \Lam^m v}_{L^2}
+
\norm{ \Lambda^{m-s+1} u}_{L^2}
\norm{ \Lambda^s \te}_{L^2}
\norm{ \Lam^m v}_{L^2} \\
& \lesssim \norm{ \nabla u}_{H^s}
\norm{ \te}_{H^m}
\norm{ \Lam^m v}_{L^2}
+
\norm{\nabla u }_{H^m}
\norm{ \Lambda^s \te}_{L^2}
\norm{ \Lam^m v}_{L^2}
\lesssim B_s A_m B_m + B_m B_s A_m,\\
M_3 
&\lesssim \norm{ \Lam^m \theta}_{L^2}
\norm{ \nabla u}_{L^{\infty}}
\norm{ \Lam^{m-1} v}_{L^2}
\lesssim \norm{ \Lam^m \theta}_{L^2}
\norm{ \nabla u}_{H^s}
\norm{ v}_{H^m}
\lesssim A_m B_s B_m.
\end{aligned}
\end{equation*}
It is easy to see that
\begin{equation*}
\cJ_{14}= 
\left|\, 2\eta\int_{\bR^2}  (v \cdot \nabla)u \cdot \nabla \theta\, \right|
\lesssim 
\norm{v}_{L^{\infty}}
\norm{ \nabla u}_{L^2}
\norm{ \nabla \te}_{L^2}
\lesssim 
\norm{v}_{H^m}
\norm{ \nabla u}_{L^2}
\norm{ \te}_{H^m}
\lesssim B_m B_s A_m.
\end{equation*}

We combine all these estimates and finally use Young's inequality to conclude that
\[
    \sum_{j=1}^{13}\cJ_k 
    \leq C(B_s A_m B_m + B_s^2A_m^2) + \frac{1}{4}B_m(t)^2
    \leq C(B_s^2 A_m^2) + \frac{1}{2}B_m(t)^2.
\]
We can hide the term $B_m(t)^2$ on the right-hand side, and \eqref{E78} is shown.
The proof is complete.
\end{proof}

\subsection{Temporal decay estimates}
In this section, we derive temporal decay estimates for solutions in Sobolev spaces of higher regularity. We shall use the temporal decay estimates derived in Section~\ref{sec_3.2}.
\begin{proposition}\label{P35}
Let $\alpha \ge 0$ and $s \in(1,2)$, and suppose
that the initial data $(u_0,v_0,\theta_0)$ belong to $\cC^{\infty}_0(\bR^2)$.  For each $m \geq s$, the global solution $(u,v,\te)$ of the Cauchy problem \eqref{E11} constructed in Proposition~\ref{key_prop} enjoys the estimate:
\begin{equation}
\label{m-1:decay}
\norm{\Lambda^{m-1} u(t)}_{L^2}
+ \norm{\Lambda^{m-1} v(t)}_{L^2}
+ \norm{\Lambda^{m-1} \theta(t)}_{L^2}
\le C (1 + t)^{- \frac{m-1}{2}}\quad \text{for all }t\geq0,
\end{equation}
where the positive constant $C>0$ which depends only on $s,m,\al,\be,\lomu, \|\mu'\|_{C^{m-1}}$, and $\|(u_0,v_0,\te_0)\|_{H^m}$.
\end{proposition}
\begin{proof} We remark that the solution $(u,v,\te)$ is smooth due to Proposition~\ref{P33}.
The following identity is obtained by substituting $m$ for 
$s$ in the expression previously used in \eqref{H^s_temp}.
\begin{equation}
\label{H^m_temportal}
\begin{split}
&\frac{1}{2} \frac{\ud}{\dt}
\left( \|\Lambda^m u\|_{L^2}^2
+ \|\Lambda^m v\|_{L^2}^2
+ \|\Lambda^m \te\|_{L^2}^2 
+ \|\Lambda^{m-1} u\|_{L^2}^2
+ \|\Lambda^{m-1} v\|_{L^2}^2
+\|\Lambda^{m-1} \te \|_{L^2}^2
-\kappa \int_{\bR^2}\Lambda^{m-1}v \cdot \Lambda^{m-1}\nabla \te
\,\right) \\
&\quad + \left(\alpha \|\Lambda^m u\|_{L^2}^2
+ \lomu\|\Lambda^m \nabla u\|_{L^2}^2 
+ \frac{\beta}{2} ||\Lambda^m v\|_{L^2}^2 
+ \alpha \|\Lambda^{m-1} u\|_{L^2}^2 
+ \lomu\|\Lambda^{m-1} \nabla u\|_{L^2}^2
+ \beta \|\Lambda^{m-1} v \|_{L^2}^2\right)
+ \frac{\kappa}{2}\|\Lambda^m \te\|_{L^2}^2
\\
&\leq - \int_{\bR^2} \Lambda^m ((u \cdot \nabla) u) \cdot \Lambda^m u
+ \int_{\bR^2} \Lambda^m (v \otimes  v) : \Lambda^m \nabla u 
- \int_{\bR^2} \Lambda^m ((u \cdot \nabla) v) \cdot \Lambda^m v
- \int_{\bR^2} \Lambda^m ((v \cdot \nabla) u) \cdot \Lambda^m v\\
&\quad - \int_{\bR^2} \Lambda^{m-1} ((u \cdot \nabla) u) \cdot \Lambda^{m-1} u
+ \int_{\bR^2} \Lambda^{m-1} (v \otimes  v) : \Lambda^{m-1} \nabla u 
- \int_{\bR^2} \Lambda^{m-1} ((u \cdot \nabla) v) \cdot \Lambda^{m-1} v\\
&
\quad - \int_{\bR^2} \Lambda^{m-1} ((v \cdot \nabla) u) \cdot \Lambda^{m-1} v  - \int_{\bR^2} \Lambda^m (u \cdot \nabla \te) \Lambda^m \te -\int_{\bR^2}\left(\Lam^m (\mu(\te)\nabla u) - \mu(\te)\Lam^m \nabla u \right):\Lam^m \nabla u\\
&
\quad - \int_{\bR^2} \Lambda^{m-1} (u \cdot \nabla \te) \Lambda^{m-1} \te -\int_{\bR^2}\left(\Lam^{m-1} (\mu(\te)\nabla u) - \mu(\te)\Lam^{m-1} \nabla u \right):\Lam^{m-1} \nabla u\\
&\quad+\kappa \int_{\bR^2} \Lam^{m-1} (u \cdot \nabla)v \cdot \Lam^{m-1}\nabla \theta
+\kappa \int_{\bR^2} \Lam^{m-1} (v \cdot \nabla)u \cdot \Lam^{m-1}\nabla \theta 
-\kappa \int_{\bR^2} \Lam^{m-1} ( u \cdot \nabla \theta) \cdot \Lam^{m-1}\dv v\\
& =: \sum_{k=1}^{15}\cJ_k.
\end{split}
\end{equation}

For each time $t>0$, we define $X_m(t)$ and $Y_m(t)$ by
\[
X_m(t) := \left( \|\Lambda^m u\|_{L^2}^2
+ \|\Lambda^m v\|_{L^2}^2
+ \|\Lambda^m \te\|_{L^2}^2 
+ \|\Lambda^{m-1} u\|_{L^2}^2
+ \|\Lambda^{m-1} v\|_{L^2}^2
+\|\Lambda^{m-1} \te \|_{L^2}^2 -\kappa\int_{\bR^2} \Lam^{m-1} v \cdot \Lam^{m-1}\nabla \theta\right)^{\frac{1}{2}}(t) 
\]
and
\[
Y_m(t) := \lambda\left(
\|\Lambda^{m+1} u\|_{L^2}^2 
+ \|\Lambda^{m} u\|_{L^2}^2
+ \alpha \|\Lambda^{m-1} u\|_{L^2}^2 
+  \|\Lambda^m v\|_{L^2}^2 
+  \|\Lambda^{m-1} v \|_{L^2}^2 + \|\Lambda^m\te\|_{L^2}^2\right)^{\frac{1}{2}}(t). 
\]
We claim that there holds
\begin{equation}
\label{claim_m}
\frac{\ud}{\dt}X_m^2 + Y_m^2 \leq \frac{C_1}{(1+t)^m} + \frac{C_2}{(1+t)^{\delta}}Y_m^2,
\end{equation}
for some fixed positive real number $\delta>0$.
To prove the claim, we estimate the nonlinear terms $\cJ_k$ in \eqref{H^m_temportal}. 
To this end, we use the temporal decay estimates of \eqref{s-1:decay}, \eqref{E57}, and several interpolations for the case $1 < s < 2$:
\begin{align*}
\|\Lambda^{s-1}(u,v,\te)\|_{L^2}&\leq C (1+t)^{-(s-1)/2},\\
\|\Lambda^{s}(u,v,\te)\|_{L^2}&\leq C (1+t)^{-(s-1)/2},\\
\|\nabla (u,v,\te)\|_{L^2}&\leq 
C\left(
\|\Lambda^{s-1}(u,v,\te)\|_{L^2}
+\|\Lambda^{s}(u,v,\te)\|_{L^2}\right) 
\leq C (1+t)^{-(s-1)/2},\\
\|(u,v,\te)\|_{L^{\infty}}&\leq 
C\left(
\|\Lambda^{s-1}(u,v,\te)\|_{L^2}
+\|\Lambda^{s}(u,v,\te)\|_{L^2}\right) 
\leq C (1+t)^{-(s-1)/2}.
\end{align*}
For simplicity, we restrict our attention to the undamped case in this proof, noting that the same argument carries over to the damped case.

Now, let us begin with the terms $\cJ_1, \cJ_2$, and $\cJ_6$. Simply, we can show that
\begin{align*}
\cJ_1
&\lesssim \norm{\nabla u}_{L^2}
\norm{\Lam^m u}_{L^4}^2
\lesssim \lt(\norm{\Lambda^{s-1} u}_{L^2}+
\norm{\Lambda^{s} u}_{L^2}\rt)
\norm{\Lam^{m+1} u}_{L^2}
\norm{\Lam^m u}_{L^2}
\lesssim (1+t)^{-(s-1)/2}Y_m^2, \\
\cJ_2
&\lesssim \norm{\Lam^m v}_{L^2} 
\norm{v}_{L^\infty} 
\norm{\Lam^{m+1} u}_{L^2} 
\lesssim 
\norm{\Lam^m v}_{L^2} 
(\norm{\Lambda^{s-1}v}_{L^2}+
\norm{\Lambda^{s}v}_{L^2})
\norm{\Lam^{m+1} u}_{L^2}
\lesssim
(1+t)^{-(s-1)/2}Y_m^2, \\
\cJ_6
&\lesssim \norm{\Lam^{m-1} v}_{L^2}
\norm{v}_{L^\infty}
\norm{\Lam^m u}_{L^2}
\lesssim \norm{\Lam^{m-1} v}_{L^2}
(\norm{\Lambda^{s-1}v}_{L^2}+\norm{\Lambda^{s}v}_{L^2})  
\norm{\Lam^{m} u}_{L^2}
\lesssim (1+t)^{-(s-1)/2}Y_m^2.
\end{align*}

To estimate the term $\cJ_3,\cJ_4,\cJ_9$, and $\cJ_{10}$, we use Lemma~\ref{L22} to obtain
\[
    \|\nabla u \|_{L^{\infty}}\lesssim \|\Lambda^{s-1}u\|_{L^2}^{a_1} \|\Lambda^{m+1}u\|_{L^2}^{1-a_1} \lesssim (1+t)^{-\frac{a_1(s-1)}{2}}Y_{m}^{1-a_2},
\]
where we denote $a_1=\frac{m-1}{m-s+2}$.
Moreover, it follows from Proposition~\ref{P33} that 
\[
    X_m(t) \lesssim A_m(t) \lesssim A_m(0).
\]
Hence, for the term $\cJ_3$, we have
\begin{align*}
\cJ_3
& 
\lesssim \norm{\nabla u}_{L^\infty}
\norm{\Lambda^m v}_{L^2}^2
+ \norm{\Lambda^m u}_{L^{\frac{2}{s-1}}}
\norm{\nabla v}_{L^{\frac{2}{2-s}}}
\norm{\Lambda^m v}_{L^2}
\\
&\lesssim (\norm{\Lambda^{s-1}u}_{L^2}^{a_1}
\norm{\Lambda^{m+1} u}_{L^2}^{1-a_1})
\norm{\Lambda^m v}_{L^2}^{1-a_1} 
\norm{\Lambda^m v}_{L^2}^{1+a_1}
+
(\norm{\Lambda^{m+1} u}_{L^2}^{2-s}
\norm{\Lambda^m u}_{L^2}^{s-1})
\norm{\Lambda^s v}_{L^2}
\norm{\Lambda^m v}_{L^2}\\
&\lesssim
(1+t)^{-\frac{a_1(s-1)}{2}}Y_m^{1-a_1}X_m^{1-a_1}Y_m^{1+a_1} + (1+t)^{-(s-1)/2} Y_m^2
\lesssim 
(1+t)^{-\frac{a_1(s-1)}{2}} Y_m^2.
\end{align*}
Proceeding similarly, we obtain the following:
\begin{align*}
\cJ_4
&\lesssim \norm{\Lam^m v}_{L^2}^2
\norm{\nabla u}_{L^\infty}
+ \norm{v}_{L^\infty}
\norm{\Lam^{m+1} u}_{L^2}
\norm{\Lam^m v}_{L^2} 
\lesssim 
(1+t)^{-\frac{a_1(s-1)}{2}} Y_m^2, \\
\cJ_9
&
\lesssim \norm{\nabla u}_{L^\infty}
\norm{\Lambda^m \theta}_{L^2}^2
+ \norm{\Lambda^m u}_{L^{\frac{2}{s-1}}}
\norm{\nabla \theta}_{L^{\frac{2}{2-s}}}
\norm{\Lambda^m \theta}_{L^2} 
\lesssim
(1+t)^{-\frac{a_1(s-1)}{2}} Y_m^2,\\
\cJ_{10}
&\lesssim \norm{\nabla \mu (\theta)}_{L^\frac{2}{2-s}}
\norm{\Lam^{m} u}_{L^\frac{2}{s-1}}
\norm{\Lam^{m+1} u}_{L^2}
+ \norm{\Lam^m (\mu(\te)-\mu(0))}_{L^2}
\norm{\nabla u}_{L^\infty}
\norm{\Lam^{m+1} u}_{L^2} \\
&\lesssim \norm{\nabla \theta}_{L^\frac{2}{s-2}}
\norm{\Lam^m u}_{L^\frac{2}{s-1}}
\norm{\Lam^{m+1} u}_{L^2}
+ (1+\|\nabla \te\|_{L^2}^{\lceil m-1 \rceil})
\norm{\Lam^m \theta}_{L^2}
\norm{\nabla u}_{L^{\infty}}
\norm{\Lam^{m+1} u}_{L^2}\\
&\lesssim 
(1+t)^{-(s-1)/2} Y_m^2 + 
(1+\norm{\theta}_{H^s}^{\lceil m-1 \rceil})(1+t)^{-\frac{a_1(s-1)}{2}}Y_m^2 \lesssim (1+t)^{-\frac{a_1(s-1)}{2}}Y_m^2,
\end{align*}
where we used \eqref{E70} in $\cJ_{10}$.

For the terms $\cJ_5, \cJ_7, \cJ_8$, and $\cJ_{11}$, we use the bounds
\[
    \norm{\nabla f}_{L^2} 
    \leq 
        \norm{\Lambda^{s-1}f}_{L^2}^{a_2}
            \norm{\Lambda^{m}f}_{L^2}^{1-a_2}
\quad \text{and} \quad
    \norm{\Lambda^{m-\frac{1}{2}}g}_{L^2} 
    \leq 
        \norm{\Lambda^{s-1}g}_{L^2}^{a_3}
            \norm{\Lambda^{m}g}_{L^2}^{1-a_3},
\]
where we denote $a_2 = \frac{m-1}{m-s+1}$ and $a_3 = \frac{1}{2(m-s+1)}$. Thus, by Young's inequality, it follows that
\begin{align*}
\cJ_5
&
\lesssim \norm{\nabla u}_{L^2}
\norm{\Lam^{m-1} u}_{L^4}^2 
\lesssim (\norm{\Lam^{s-1} u}_{L^2}^{a_2}
\norm{\Lam^{m} u}_{L^2}^{1-a_2})
(\norm{\Lam^{s-1} u}_{L^2}^{a_3}
\norm{\Lam^{m} u}_{L^2}^{1-a_3})^2
\\
&\le C
\norm{\Lam^{s-1} u}_{L^2}^{\frac{m}{m-s+1}}
\norm{\Lam^{m} u}_{L^2}^{3-\frac{m}{m-s+1}}
\le C 
\norm{\Lam^{s-1} u}_{L^2}^{\frac{2m}{s-1}} + \frac{\underline{\mu}}{8}
\norm{\Lam^{m} u}_{L^2}^{2}
\le C
(1+t)^{-m} + \frac{1}{8}Y_m^2.
\end{align*}
Using an analogous computation, we arrive at the following conclusion:
\begin{align*}
\cJ_7 &\lesssim
\norm{\nabla u}_{L^2}
\norm{\Lam^{m-1} v}_{L^4}^2+
\norm{\Lam^{m-1} u}_{L^4}
\norm{\nabla v}_{L^{2}}
\norm{\Lam^{m-1} v}_{L^4} \le C
(1+t)^{-m} + \frac{1}{8}Y_m^2,\\
\cJ_8 & \lesssim
\norm{\nabla u}_{L^2}
\norm{\Lam^{m-1} v}_{L^4}^2+
\norm{v}_{L^{\infty}}
\norm{\Lam^{m} u}_{L^2}
\norm{\Lam^{m-1} v}_{L^2} \le C
(1+t)^{-m} + \frac{1}{8}Y_m^2 + C(1+t)^{-(s-1)/2}Y_m^2,\\
\cJ_{11}
&\lesssim
\norm{\nabla u}_{L^2} 
\norm{\Lam^{m-1} \te}_{L^4}
\norm{\Lam^{m-1} \te}_{L^4} 
+
\norm{\Lam^{m-1} u}_{L^4}
\norm{\nabla \te}_{L^2}
\norm{\Lam^{m-1} \te}_{L^4} 
\leq C(1+t)^{-m}+\frac{1}{8}Y_m^2.
\end{align*}

To bound $\cJ_{12}, \cJ_{13}, \cJ_{13}$, and $\cJ_{15}$, we use the bounds
\[
    \norm{\nabla f}_{L^{\frac{2}{2-s}}} \lesssim 
        \norm{\Lam^s f}_{L^2} \lesssim
        \norm{\nabla f}_{L^2}^{a_4}
            \norm{\Lambda^{m}f}_{L^2}^{1-a_4}
\quad \text{and} \quad
    \norm{\Lambda^{m-1}g}_{L^\frac{2}{s-1}} 
    \lesssim 
        \norm{\Lambda^{m-s+1}g}_{L^2}\lesssim
        \norm{\nabla g}_{L^2}^{1-a_4}
            \norm{\Lambda^{m}g}_{L^2}^{a_4},
\]
where we denote $a_4 = \frac{m-s}{m-1}$.
For the term $\cJ_{12},$ we apply \eqref{E70} and \eqref{E71} to obtain
\begin{align*}
\cJ_{12}& \lesssim 
\norm{\nabla \mu (\theta)}_{L^\frac{2}{2-s}}
\norm{\Lam^{m-1} u}_{L^{\frac{2}{s-1}}}
\norm{\Lam^m u}_{L^2} 
+
\norm{\Lam^{m-1} \mu (\theta)}_{L^\frac{2}{s-1}}
\norm{\nabla u}_{L^\frac{2}{2-s}}
\norm{\Lam^m u}_{L^2} \\
&\lesssim
\norm{\nabla \theta}_{L^\frac{2}{2-s}}
\norm{\Lam^{m-s+1} u}_{L^2}
\norm{\Lam^m u}_{L^2} 
+
\norm{\Lam^{m-s+1} (\mu(\te)-\mu(0))}_{L^2}
\norm{\Lam^s u}_{L^2}
\norm{\Lam^m u}_{L^2}\\
&\lesssim
\norm{\Lam^{s} \theta}_{L^2}
\norm{\Lam^{m-s+1} u}_{L^2}
\norm{\Lam^m u}_{L^2} 
+
(1+\|\nabla \te\|_{L^2}^{\lceil m-s \rceil})
\norm{\Lam^{m-s+1} \theta}_{L^2}
\norm{\Lam^s u}_{L^2}
\norm{\Lam^m u}_{L^2}\\
&\lesssim
(\norm{\nabla \theta}_{L^2}^{a_4}
\norm{\Lam^m \theta}_{L^2}^{1-a_4})
(\norm{\nabla u}_{L^2}^{1-a_4}
\norm{\Lam^{m} u}_{L^2}^{a_4})
\norm{\Lam^m u}_{L^2} 
+
(\norm{\nabla \theta}_{L^2}^{1-a_4}
\norm{\Lam^m \theta}_{L^2}^{a_4})
(\norm{\nabla u}_{L^2}^{a_4}
\norm{\Lam^{m} u}_{L^2}^{1-a_4})
\norm{\Lam^m u}_{L^2}\\
& \lesssim (1+t)^{-(s-1)/2}Y_m^2.
\end{align*}
Similarly, it follows that
\[
\cJ_{14}
\lesssim 
\norm{v}_{L^\infty}
\norm{\Lam^m u}_{L^2}
\norm{\Lam^m \theta}_{L^2} + 
\norm{\Lam^{m-1} v}_{L^{\frac{2}{s-1}}}
\norm{\nabla u}_{L^{\frac{2}{2-s}}}
\norm{\Lam^{m} \te}_{L^2}
\lesssim (1+t)^{-(s-1)/2}Y_m^2. 
\]
The term $\cJ_{13}+ \cJ_{15}$ can be represented, similarly to \eqref{E47a}, as follows:
\begin{align*}
\cJ_{13}+\cJ_{15}=
&\int_{\bR^2} \left( \Lam^{m-1} (u_i \p_i v_j) 
- u_i \p_i \Lam^{m-1} v_j \right) \p_j \Lam^{m-1} \theta
-  \int_{\bR^2} \left( \Lam^{m-1} (u_i \p_i \theta)
- u_i \p_i \Lam^{m-1} \theta \right) \Lam^{m-1} \p_j v_j \\
&\quad + \int_{\bR^2}  (\p_i \Lam^{m-1} \theta) (\p_j u_i)(\Lam^{m-1} v_j)
=: R_1 + R_2 + R_3.
\end{align*}
Then, these terms can also be computed similarly by
\begin{align*}
    R_1
&\lesssim
\norm{\nabla u}_{L^{\frac{2}{2-s}}}
\norm{\Lam^{m-1} v}_{L^{\frac{2}{s-1}}}
\norm{\Lam^{m} \te}_{L^2}
+
\norm{\Lam^{m-1} u}_{L^{\frac{2}{s-1}}}
\norm{\nabla v}_{L^{\frac{2}{2-s}}}
\norm{\Lam^{m} \te}_{L^2}
\lesssim (1+t)^{-(s-1)/2}Y_m^2,\\
R_2
&\lesssim
\norm{\nabla u}_{L^{\frac{2}{2-s}}}
\norm{\Lam^{m-1} \te}_{L^{\frac{2}{s-1}}}
\norm{\Lam^{m} v}_{L^2}
+
\norm{\Lam^{m-1} u}_{L^{\frac{2}{s-1}}}
\norm{\nabla \te}_{L^{\frac{2}{2-s}}}
\norm{\Lam^{m} v}_{L^2}
\lesssim (1+t)^{-(s-1)/2}Y_m^2,\\
    R_3
&\lesssim
\norm{\Lam^{m} \te}_{L^2}
\norm{\nabla u}_{L^{\frac{2}{2-s}}}
\norm{\Lam^{m-1} v}_{L^{\frac{2}{s-1}}}
\lesssim (1+t)^{-(s-1)/2}Y_m^2.
\end{align*}
Combining these together, the nonlinear terms $\sum_{k=1}^{15}\cJ_{k}$ are bounded in the form of $\frac{C_1}{(1+t)^m} + \frac{C_2}{(1+t)^{\delta}}Y_m^2 + \frac{1}{2}Y_m^2$, and thus the claim~\eqref{claim_m} is shown.

Now, we are ready to establish the temporal decay estimates \eqref{m-1:decay}.
Assuming \eqref{claim_m}, i.e.
\[
\frac{\ud}{\dt}X_m^2 + Y_m^2 \leq \frac{C_1}{(1+t)^m} + \frac{C_2}{(1+t)^{\delta}}Y_m^2\quad \text{for some } \delta>0,
\]
we choose a time $T_1>0$ such that $\frac{C_2}{(1+T_1)^{\delta}}<\frac{1}{2}$. 
By using Gagliardo--Nirenberg inequality in Lemma~\ref{L22}, it can be easily shown that
\[
    2C_0(X_m^2)^{\frac{m}{m-1}}\leq Y_m^2,
\]
for some positive constant $C_0$.
The choice of $T_1$ leads to that
\[
\frac{\ud}{\dt}X_m^2 + C_0(X_m^2)^{\frac{m}{m-1}} \leq \frac{C_1}{(1+t)^m},    \quad \text{if }t\geq T_1.
\]
To simplify notation, we set $y(t):=X_m(t)^2$. 
Note that
\begin{align*}
        \lt((1+t)^{m}y(t)\rt)'&
        =m(1+t)^{m-1}y(t) + (1+t)^m\lt(y'(t) + C_0(y(t))^{\frac{m}{m-1}}\rt)-C_0(1+t)^my(t)^{\frac{m}{m-1}}\\
        & \leq (1+t)^my(t)\lt(\frac{m}{1+t}-C_0(y(t))^{\frac{1}{m-1}}\rt)+C_1.
\end{align*}
Now, we divide into two cases:
    \begin{enumerate}
        \item [(i)] $\frac{m}{1+t}-C_0(y(t))^{\frac{1}{m-1}}\geq 0$ : Then, this implies that $y(t) \leq \lt(\frac{m}{C_0(1+t)}\rt)^{m-1}$ and this yields that
        \begin{equation*}
            \lt((1+t)^{m}y(t)\rt)'\leq (1+t)^my(t)\lt(\frac{m}{1+t}-C_0(y(t))^{\frac{1}{m-1}}\rt)+C_1 \leq \frac{m^m}{C_0^{m-1}}+C_1.
        \end{equation*}
        \item [(ii)] $\frac{m}{1+t}-C_0(y(t))^{\frac{1}{m-1}} < 0$ : Then, it is clear that
               \begin{equation*}
            \lt((1+t)^{m}y(t)\rt)' \leq C_1.
        \end{equation*}
    \end{enumerate}
On the other hand, we can obtain the uniform bound of $X_m(t)$ in $[0,T_1]$ by
\begin{equation*}
    \sup_{t\in[0,T_1]}X_m(t) \lesssim \sup_{t\in[0,\infty)}A_m(t) <C_3\epsilon.
\end{equation*}
Combining these, we conclude that
\begin{equation*}
    X_m(t)^2 \leq C_4(1+t)^{-m+1},
\end{equation*}
where the constant $C_4$ can be explicitly bounded by $\max\lt(\frac{m^m}{C_0^{m-1}}+C_1+(C_3\ep)^2,\, (C_3\ep)^2(1+T_1)^{m-1} \rt)$. Notice that the constant $C_4$ depends solely on the parameters $s,\al,\be,\lomu, \|\mu'\|_{C^{m-1}}$ and $\|(u_0,v_0,\te_0)\|_{H^m}$.

We arrive at
\[
    \norm{\Lambda^{m-1} u(t)}_{L^2}
+ \norm{\Lambda^{m-1} v(t)}_{L^2}
+ \norm{\Lambda^{m-1} \theta(t)}_{L^2}
\leq 2X_m(t)
\le C (1 + t)^{- \frac{m-1}{2}}\quad \text{for all }t\geq0.
\]
\end{proof}

We close this section by stating a simple but useful lemma.
\begin{lemma}\label{Lem:ODE_2}
Let $y(t)$ be a nonnegative differentiable function that satisfies
\[
    y'(t) + ky(t)\leq \frac{C}{(1+t)^{a}},
\]
for some positive constants $k,a$ and $C$.
Then, there holds
\[
    y(t) \lesssim \frac{1}{(1+t)^a}.
\]
\end{lemma}
\begin{proof}
    By Gr\"onwall's inequality, one can simply show that
\[
        y(t) \leq e^{-kt}y(0) + \int_0^t e^{-k(t-\tau)}\frac{C}{(1+\tau)^a}\ud \tau.
\]
Then, the result comes from
\begin{align*}
    \int_0^t e^{-k(t-\tau)}\frac{C}{(1+\tau)^a}\ud \tau & =\int_0^{t/2} e^{-k(t-\tau)}\frac{C}{(1+\tau)^{a}}\ud \tau
    +\int_{t/2}^t e^{-k(t-\tau)}\frac{C}{(1+\tau)^{a}}\ud \tau\\
    & \leq   
    e^{-kt/2}\int_0^{t/2} \frac{C}{(1+\tau)^{a}}\ud \tau
    + \frac{C}{(1+\frac{t}{2})^{a}}\int_{t/2}^t e^{-k(t-\tau)}\ud \tau =\mathcal{O}\lt(\frac{1}{(1+t)^{a}}\rt).
\end{align*}
\end{proof}


\subsection{Proof of main theorems}
\begin{proof}[Proof of Theorem ~ \ref{main_thm:no_damping} and Theorem ~ \ref{main_thm:damped}]
It suffices to consider the case $1<s<2$, since all the remaining cases are included in this range as well. Indeed, one can choose $\epsilon_1$ and $\epsilon_2$ as $\varepsilon^*$ in Proposition~\ref{key_prop}.
In the undamped case $\alpha = 0$, the estimates \eqref{E15} directly follows from Proposition~\ref{P33}. 
In the damped case $\alpha > 0$, we recall $L^2$-energy estimate
\begin{equation}
\label{E16a}
\frac{1}{2} \frac{\ud}{\dt}
\left(\|u\|_{L^2}^2 
+ \|v\|_{L^2}^2
+ \|\theta\|_{L^2}^2 \right) 
+ \int_{\mathbb{R}^2}\mu(\theta)|\nabla u|^2
+ \alpha \|u\|_{L^2}^2 +\beta\|v\|_{L^2}^2 = 0.
\end{equation}
Then the estimate \eqref{E16} follows from the Proposition ~ \ref{P33} together with the $L^2$-energy estimate \eqref{E16a}.

Now, it only remains to establish the temporal decay estimates \eqref{E44} and \eqref{E45}. For any positive number $\gamma >0$, by Proposition ~ \ref{P35}, we have
\begin{equation}
\label{gamma:decay}
\norm{\Lambda^{\gamma} u(t)}_{L^2}
+ \norm{\Lambda^{\gamma} v(t)}_{L^2}
+ \norm{\Lambda^{\gamma} \theta(t)}_{L^2}
\le C (1 + t)^{- \frac{\gamma}{2}}\quad \text{for all }t\geq0.
\end{equation}
The decay rate of $u$ and $v$ is shown to be further improvable.

We multiply $v$ to $\eqref{E11}_2$ and integrate over $\bR^2$ to obtain
\begin{align*}
    \frac{1}{2}\frac{\ud}{\dt}  \|\Lambda^{\gamma}v\|_{L^2}^2
   + \beta \|\Lambda^{\gamma}v\|_{L^2}^2
  &=  \int_{\bR^2}\Lambda^{\gamma}\nabla \te\cdot \Lambda^{\gamma}v -\int_{\bR^2}\Lambda^{\gamma}(v\cdot \nabla) u\cdot \Lambda^{\gamma}v -\int_{\bR^2}\Lambda^{\gamma}(u\cdot \nabla) v\cdot \Lambda^{\gamma}v\\
  &=: \cJ_1 + \cJ_2 + \cJ_3. 
\end{align*}
By Young's inequality with \eqref{gamma:decay}, we see that
\begin{align*}
        \cJ_1 
        &\leq \|\Lambda^{\gamma+1}\te\|_{L^2}
        \|\Lambda^{\gamma}v\|_{L^2}
        \leq C(1+t)^{-(\gamma+1)} + 
        \frac{\beta}{5}\|\Lambda^{\gamma}v\|_{L^2}^2,\\
        \cJ_2 
        &\leq \|\Lambda^{\gamma}v\|_{L^2}^2
        \|\nabla u\|_{L^{\infty}} + 
        \|v\|_{L^{\infty}}
        \|\Lambda^{\gamma}v\|_{L^2}
        \|\Lambda^{\gamma+1}u\|_{L^2}\\
        &\leq C(1+t)^{-1}\|\Lambda^{\gamma}v\|_{L^2}^2 + 
        C(1+t)^{-(\gamma+2)} + 
        \frac{\beta}{5}\|\Lambda^{\gamma}v\|_{L^2}^2,
        \\
        \cJ_3
        &\leq  \left|\int_{\bR^2}(\Lambda^{\gamma}(u\cdot \nabla)v-(u\cdot \nabla)\Lambda^{\gamma}v) \cdot \Lambda^{\gamma}v \, \right|\\
        &\leq \|\Lambda^{\gamma}v\|_{L^2}^2
        \|\nabla u\|_{L^{\infty}} + 
        \|\nabla v\|_{L^{\infty}}
        \|\Lambda^{\gamma}v\|_{L^2}
        \|\Lambda^{\gamma}u\|_{L^2}
        \\
        &\leq C(1+t)^{-1}\|\Lambda^{\gamma}v\|_{L^2}^2 + 
        C(1+t)^{-(\gamma+2)} + 
        \frac{\beta}{5}\|\Lambda^{\gamma}v\|_{L^2}^2.
\end{align*}
Hence, if $t$ is sufficiently large, it follows that
\[
\frac{1}{2}\frac{\ud}{\dt}  \|\Lambda^{\gamma} v\|_{L^2}^2
+ \frac{\beta}{5} \|\Lambda^{\gamma} v\|_{L^2}^2 \lesssim  (1+t)^{-(\gamma+1)} + (1+t)^{-(\gamma+2)} \lesssim (1+t)^{-(\gamma+1)}.  
\]
Applying Lemma~\ref{Lem:ODE_2}, we obtain the decay estimate of $v$ by
\[
    \|\Lam^{\gamma} v(t)\|_{L^2}\leq C(1+t)^{-(\gamma+1)/2}.
\]

Next, we multiply $\Lambda^{2\gamma}u$ to $\eqref{E11}_1$ and integrate over $\bR^2$ to obtain
\begin{align*}
   &\frac{1}{2}\frac{\ud}{\dt}  \|\Lambda^{\gamma}
 u\|_{L^2}^2
   + \lomu \|\Lambda^{\gamma}\nabla u\|_{L^2}^2
    + \al \|\Lambda^{\gamma} u \|_{L^2}^2\\
    &\quad \leq \lt| \int_{\bR^2}  \Lambda^{\gamma}(u \cdot \nabla) u \cdot \Lambda^{\gamma}u \,\rt|  
    + \lt| \int_{\bR^2} \lt(\Lambda^{\gamma}(\mu(\theta)\nabla u) - \mu(\te)\Lambda^{\gamma}\nabla u\rt) : \Lambda^{\gamma}\nabla u\, \rt|
    + \lt| \int_{\bR^2} \Lambda^{\gamma}\dv (v \otimes v) \Lambda^{\gamma} u \, \rt|\\
    & \quad =: \cI_1 + \cI_2 + \cI_3.
\end{align*}
Then, one can obtain
\begin{align*}
\cI_1 &
\lesssim 
\|\nabla u\|_{L^2}
\|\Lambda^{\gamma} u\|_{L^4}^2
\lesssim 
\|\nabla u\|_{L^2}
\|\Lambda^{\gamma} u\|_{L^2}
\|\Lambda^{\gamma+1} u\|_{L^2}
\lesssim 
(1+t)^{-1/2}
\|\Lambda^{\gamma} u\|_{L^2}
\|\Lambda^{\gamma+1} u\|_{L^2}, \\
\cI_2 & 
\lesssim 
(1+\|\nabla\te\|_{L^2}^{\gamma})
\|\Lambda^{\gamma} \te\|_{L^2}
\|\nabla u \|_{L^{\infty}} 
\|\Lambda^{\gamma+1} u\|_{L^2} 
+ 
\|\nabla \te\|_{L^4}
\|\Lambda^{\gamma} u \|_{L^4} 
\|\Lambda^{\gamma+1} u\|_{L^2} \\
&\lesssim 
(1+t)^{-\gamma/2}
\|\nabla u \|_{L^{\infty}} 
\|\Lambda^{\gamma+1} u\|_{L^2} 
+ 
\|\Lam^{\frac{3}{2}} \te\|_{L^2}
\|\Lambda^{\gamma+\frac{1}{2}} u \|_{L^2} 
\|\Lambda^{\gamma+1} u\|_{L^2} \\ 
& \lesssim
(1+t)^{-\gamma}\|\nabla u \|_{L^{\infty}}^2 + \frac{\lomu}{2}\|\Lambda^{\gamma+1}u\|_{L^2}^2
+ (1+t)^{-\frac{3}{4}} (\|\Lambda^{\gamma}u\|_{L^2}
+\|\Lambda^{\gamma+1}u\|_{L^2})
\|\Lambda^{\gamma+1}u\|_{L^2}, \\
\cI_3 & 
\lesssim
\|\Lambda^{\gamma+1} v\|_{L^2}
\|v\|_{L^{\infty}}
\|\Lam^m u\|_{L^2}
\lesssim (1+t)^{-(\gamma+4)/2}
\|\Lam^m u\|_{L^2}
\leq C(1+t)^{-(\gamma+4)}
+ \frac{\alpha}{4} \|\Lam^m u\|_{L^2}^2.
\end{align*}
If $t$ is large enough, we have
\begin{equation}
\label{put:iteration}
    \frac{\ud}{\dt}  \|\Lambda^{\gamma} u\|_{L^2}^2
    + \frac{\al}{2} \|\Lambda^{\gamma} u \|_{L^2}^2
    \lesssim (1+t)^{-\gamma}\|\nabla u \|_{L^{\infty}}^2 + (1+t)^{-(\gamma+4)}.
\end{equation}
By observing $\|\nabla u \|_{L^{\infty}}^2 \lesssim(1+t)^{-2}$, we apply Lemma~\ref{Lem:ODE_2} to conclude that
\begin{equation}
\label{put:1}
     \|\Lambda^{\gamma} u(t)\|_{L^2} \leq C(1+t)^{-(\gamma+2)/2}. 
\end{equation}
In fact, the estimate \eqref{put:1} further implies $\|\nabla u \|_{L^{\infty}}^2 \lesssim(1+t)^{-4}$. We combine this with \eqref{put:iteration}
to deduce that 
\[
     \frac{\ud}{\dt}  \|\Lambda^{\gamma} u\|_{L^2}^2
    + \frac{\al}{2} \|\Lambda^{\gamma} u \|_{L^2}^2
    \lesssim (1+t)^{-(\gamma+4)}.
\]
We apply Lemma~\ref{Lem:ODE_2} again, and these yield the desired result \eqref{E44} and \eqref{E45}. The proof is complete.
\end{proof}

\appendix

\section{Auxiliary estimates} \label{Appen_B}

\begin{lemma}[$L^{\infty}$-bound, $\Lambda^s$-type interpolation.] \label{L23} Let $d$ be a positive integer.
    \begin{enumerate}
        \item[(i)] Let $s_1,s_2$ and $s$ be nonnegative real numbers satisfying $0\leq s_1<s<s_2$.
    Suppose $u \in H^{s_2}(\mathbb{R}^d)$. Then, it holds that
\[
        \|\Lambda^su\|_{L^2(\bR^d)} \leq  \|\Lambda^{s_1}u\|_{L^2(\bR^d)}^{\frac{s_2-s}{s_2-s_1}} \|\Lambda^{s_2}u\|_{L^2(\bR^d)}^{\frac{s-s_1}{s_2-s_1}}.
\]
    \item[(ii)]  
    Let $s\geq0$, and let $s_1,s_2$ be two nonnegative real numbers satisfying $0\leq s_1<\frac{d}{2}$ and $\frac{d}{2}+s<s_2$.
Suppose $u \in H^{s_2}(\mathbb{R}^d)$. Then, there exists a constant $C=C(d,s,s_1,s_2)>0$  such that
\[
        \|\Lambda^s u\|_{L^{\infty}(\bR^d)} \leq C \|\Lambda^{s_1}u\|_{L^2(\bR^d)}^{\frac{s_2-s-\frac{d}{2}}{s_2-s_1}} \|\Lambda^{s_2}u\|_{L^2(\bR^d)}^{\frac{\frac{d}{2}+s-s_1}{s_2-s_1}}.
\]
In particular, there holds
\[
\|u\|_{L^{\infty}(\bR^d)} \leq C \|\Lambda^{s_1}u\|_{L^2(\bR^d)}^{\frac{s_2-\frac{d}{2}}{s_2-s_1}} \|\Lambda^{s_2}u\|_{L^2(\bR^d)}^{\frac{\frac{d}{2}-s_1}{s_2-s_1}}.
\]
    \end{enumerate}
\end{lemma}

\begin{proof}
\begin{enumerate}
\item[(i)]
We define the exponents $p$ and $q$ by
    \begin{equation*}
        \frac{1}{p} = \frac{s_2-s}{s_2-s_1}, \quad \frac{1}{q} = \frac{s-s_1}{s_2-s_1}.
    \end{equation*}
Then the following relations hold:
    \begin{equation*}
        \frac{1}{p}+\frac{1}{q}=1, \quad \frac{s_1}{p}+\frac{s_2}{q}=s.
    \end{equation*}
Applying H\"older's inequality, we deduce
\begin{equation*}
    \int_{\bR^d}|\xi|^{2s}|\hat{u}(\xi)|^2 \,\ud \xi= \int_{\bR^d}\lt(|\xi|^{\frac{2s_1}{p}}|\hat{u}|^{\frac{2}{p}}\rt) \lt(|\xi|^{\frac{2s_2}{q}}|\hat{u}|^{\frac{2}{q}}\rt)\,\ud \xi \leq \left(\int_{\bR^d}|\xi|^{2s_1}|\hat{u}|^{2}\right)^{\frac{1}{p}}\left(\int_{\bR^d}|\xi|^{2s_2}|\hat{u}|^{2}\right)^{\frac{1}{q}}.
\end{equation*}
\item[(ii)]
As a direct consequence of {[Lemma 2.1, \cite{JKL2022}]}, we obtain that for $s \ge0$ and $p,q > 0$,
\[
\norm{\Lam^s u}_{L^\infty (\bR^d)}
\le C \norm{|\xi|^s \hat{u}}_{L^1 (\bR^d)}
\le C \norm{|\xi|^{s + \frac{d}{2} + p} \hat{u}}_{L^2 (\bR^d)}^\frac{q}{p+q}
\norm{|\xi|^{s + \frac{d}{2} - q} \hat{u}}_{L^2 (\bR^d)}^\frac{p}{p+q}.
\]
Choosing $p = s_2 - (\frac{d}{2} + s)$ and $q = \frac{d}{2} + s - s_1$ follows the desired result.
\end{enumerate}
\end{proof}

\begin{lemma}
Let $s \ge 1$.
Suppose that $h:\Omega \rightarrow \bR$ is a smooth real-valued function on $\Omega$ with $h(0)=0$. We also assume that $f : \bR^{d} \rightarrow \Omega$ is a continuous function with $\text{Im}(f) \subset \Omega_1$, $\overline{\Omega}_1 \subset \subset \Omega$, and $f \in L^\infty \cap H^s$. We denote the composition of $h$ and $f$ by $h(f)$. Then, the composition $h (f)$ belongs to $H^s (\bR^d)$, and it holds that
\begin{equation}
\label{E92}
\norm{ \Lam^s h (f)}_{L^2(\bR^d)}
\le C
\norm{h'}_{C^{\lceil s-1 \rceil} (\overline{ \Omega_1})} 
(1 + \norm{\nabla f}_{L^2(\bR^d)}^{\lceil s-1 \rceil})
\norm{ \Lam^s f}_{L^2(\bR^d)},
\end{equation}
where $C$ is a constant depending only on $s\geq1$ and $d$.
\end{lemma}
\begin{proof}
We define $h_0(x) := h(x) - h'(0)x$. Then, we can write $h(f) = h_0(f) + h'(0)f$. Clearly, it holds that $\norm{ \Lam^s h'(0) (f)}_{L^2}
\le |h'(0)| \norm{ \Lam^s f}_{L^2} \leq 
\norm{h'}_{C^{s-1}} 
\norm{ \Lam^s f}_{L^2}$.
Therefore, we only need to investigate $h_0$. Note that $h_0(0) = h_0'(0) = 0$. Without loss of generality, it  is sufficient to assume $h'(0)=0$.

We assume first that $f \in \cC^{\infty}_c(\bR^d)$. 
For the case where $s$ is an integer, we can proceed with mathematical induction combining the chain rule, the interpolation inequality \eqref{E23}, and the Gagliardo--Nirenberg inequality \eqref{E21}. That is, for a positive integer $k$, we claim that
\begin{equation}
\label{claim_apen.b}
\norm{ \nabla ^k h (f)}_{L^2(\bR^d)}
\le C
\norm{h'}_{C^{k-1} (\overline{ \Omega_1})} 
(1 + \norm{\nabla f}_{L^2(\bR^d)}^{k-1})
\norm{ \nabla ^k f}_{L^2(\bR^d)}.
\end{equation}
For $k=1$, the statement follows directly from $\nabla (h(f)) = h'(f)\nabla f$. For $k=2$, we can combine $\nabla^2 (h(f)) = h''(f)(\nabla f:\nabla f) + h'(f)\nabla^2 f$ with $\|\nabla f\|_{L^4}^2 \leq \|\nabla f\|_{L^2}\|\nabla^2 f\|_{L^2}$, to show that \eqref{claim_apen.b} holds. Assuming \eqref{claim_apen.b} holds for $k\geq2$, one can use Gagilrado-Nirenberg inequality to deduce that
\begin{align*}
\norm{ \nabla ^{k+1} h (f)}_{L^2} &= 
\norm{ \nabla ^{k} (h' (f)\nabla f)}_{L^2} 
\leq \norm{ \nabla ^{k} h' (f)}_{L^2} 
\norm{\nabla f}_{L^\infty}\\ 
& \leq C\norm{h''}_{C^{k-1} (\overline{ \Omega_1})} 
(1 + \norm{\nabla f}_{L^2}^{k-1})
\norm{ \nabla ^k f}_{L^2}\norm{\nabla f}_{L^\infty}\\
& \leq C\norm{h'}_{C^{k} (\overline{ \Omega_1})} 
(1 + \norm{\nabla f}_{L^2}^{k-1})
\norm{ \nabla ^{k+1} f}_{L^2}\norm{\nabla f}_{L^2}\\
&\leq C\norm{h'}_{C^{k} (\overline{ \Omega_1})} 
(1 + \norm{\nabla f}_{L^2}^{k})
\norm{ \nabla ^{k+1} f}_{L^2}.
\end{align*}
Thus, by mathematical induction, \eqref{claim_apen.b} is shown to hold for all natural numbers.

Suppose that $s$ is not an integer, say $s=k+\delta$ for some $k \in \bN$ and $0<\delta<1$. 
Then, 
\begin{equation*}
\norm {\Lambda^s h(f)}_{L^2}
=\norm {\Lambda^{k-1+\delta} \nabla (h(f))}_{L^2} 
=\norm {\Lambda^{k-1+\delta} (h'(f)\nabla f)}_{L^2}. 
\end{equation*}
Note that the assumption $h'(0)=0$ implies that $h'(f) \in \cC^{\infty}_c(\bR^d)$. Hence, it follows from \eqref{claim_apen.b} that
\begin{align*}
\norm {\Lambda^{k-1+\delta} (h'(f)\nabla f)}_{L^2}
&\lesssim \norm {\Lambda^{k-1+\delta} h'(f)}_{L^{\frac{2}{\delta}}}
\norm {\nabla f}_{L^{\frac{2}{1-\delta}}}
+
\norm {h'(f)}_{L^{\infty}}
\norm {\Lambda^{k-1+\delta} \nabla f}_{L^2}\\
&\lesssim 
\norm {\Lambda^{k} h'(f)}_{L^2}
\norm {\Lambda^{1+\delta} f}_{L^2}
+
\norm {h'}_{L^{\infty}}
\norm {\Lambda^s f}_{L^2}\\
&\lesssim 
\norm{h''}_{C^{k-1}} 
(1 + \norm{\nabla f}_{L^2}^{k-1})
\norm{ \nabla ^k f}_{L^2}
\norm {\Lambda^{1+\delta} f}_{L^2}
+
\norm {h'}_{L^{\infty}}
\norm {\Lambda^s f}_{L^2}\\
&\lesssim 
\norm{h'}_{C^k}
\left( 
(1 + \norm{\nabla f}_{L^2}^{k-1})
\norm{ \Lambda^{k+\delta} f}_{L^2}
\norm {\nabla f}_{L^2}
+
\norm {\Lambda^s f}_{L^2}
\right)\\
&\lesssim 
\norm{h'}_{C^k}
(1 + \norm{\nabla f}_{L^2}^k)
\norm {\Lambda^s f}_{L^2}.
\end{align*}
This shows that \eqref{E92} holds for $f \in C^{\infty}_c(\bR^d)$. Now, for general $f \in L^{\infty}\cap H^s$, we can use standard approximation argument to conclude the proof.
\end{proof}

\section*{Acknowledgement}

We are grateful to Dr. Jaeyong Shin for his valuable discussions and insights.

This research of H. In was supported by the National Research Foundation of Korea(NRF) grant funded by the Korea government(MSIT) (No. RS-2024-00360798) and Basic Science Research Program through the  National Research Foundation of Korea(NRF) funded by the Ministry of Education(RS-2021-NR060141). 

This research of J. Kim was supported by the National Research Foundation of Korea(NRF) grant funded by the Korea government(MSIT) (No. RS-2024-00360798)


\end{document}